\documentclass[12pt,psfig,reqno]{amsart}
\usepackage{mathrsfs}
\usepackage{txfonts}
\usepackage{cite}
\usepackage{amsmath}
\usepackage{bbm}
\usepackage{}
\usepackage{amssymb}
\usepackage{amsfonts}
\usepackage{amscd}
\usepackage{amssymb,amsfonts,latexsym,enumerate}
\usepackage{graphics,verbatim}
\usepackage{graphicx}
\usepackage[usenames]{color} 
\makeatletter

\renewcommand\section{\@startsection {section}{1}{\z@}%
                                   {-3.5ex \@plus -1ex \@minus -.2ex}%
                                   {2.3ex \@plus.2ex}%
                                   {\centering\normalfont\bf}}

 \numberwithin{equation}{section}
\numberwithin{equation}{section}

\numberwithin{equation}{section}

\theoremstyle{plain}
\newtheorem{thm}{Theorem}[section]
\newtheorem{lemma}[thm]{Lemma}
\newtheorem{pro}[thm]{Proposition}

\newtheorem{ex}[thm]{Example}
\newtheorem{de}[thm]{Definition}
\newtheorem{re}[thm]{Remark}
\newtheorem*{co*}{Conjecture}
\newtheorem*{thm*}{Theorem}

\begin{document}
\title{Spectrality of generalized  Sierpinski-type self-affine measures}
\author{Jing-Cheng Liu, Ying Zhang, Zhi-Yong Wang and Ming-Liang Chen$^*$}
\address{Key Laboratory of High Performance Computing and Stochastic Information Processing (Ministry of Education of China), College of Mathematics and Statistics, Hunan Normal University, Changsha, Hunan 410081, China}
\email{jcliu@hunnu.edu.cn}
\email{zhangying19952020@163.com}
\address{College of Mathematics and Computational Science, Hunan First Normal University, Changsha, Hunan 410205, P. R. China }
\email{wzyzzql@163.com}
\address{School of Mathematics, Sun Yat-Sen University, Guangzhou 510275, P. R. China}
\email{mathcml@163.com}

\date{\today}
\thanks{ The research is supported in part by the NNSF of China (Nos. 12071125, 12001183, 11831007 and 11971500), the Hunan Provincial NSF (Nos. 2019JJ20012 and 2020JJ5097),
the SRF of Hunan Provincial Education Department (Nos.17B158 and 19B117).\\
$^*$Corresponding author.}

\begin{abstract}
For an expanding integer matrix $M\in M_2(\mathbb{Z})$ and an integer digit set
$D=\{(0,0)^t,(\alpha_1,\alpha_2)^t,(\beta_1,\beta_2)^t\}$
with $\alpha_1\beta_2-\alpha_2\beta_1\neq0$, let $\mu_{M,D}$ be the Sierpinski-type self-affine measure  defined by $\mu_{M,D}(\cdot)=\frac{1}{3}\sum_{d\in D}\mu_{M,D}(M(\cdot)-d)$. In \cite{Chen-Liu_2019,Liu-Wang}, the authors separately investigated the  spectral property of the measure $\mu_{M,D}$ in the case of $\det(M)\notin 3\mathbb{Z}$ or $\alpha_1\beta_2-\alpha_2\beta_1\notin 3\mathbb{Z}$. In this paper, we consider the remaining case where $\det(M)\in 3\mathbb{Z}$ and $\alpha_1\beta_2-\alpha_2\beta_1\in 3\mathbb{Z}$, and give the necessary and sufficient conditions for $\mu_{M,D}$ to be a spectral measure.  This completely settles the spectrality of the Sierpinski-type self-affine measure $\mu_{M,D}$.
\end{abstract}

\maketitle

\section{\bf Introduction\label{sect.1}}

Let $\mu$ be a Borel probability measure with compact support on $\mathbb{R}^n$. We call it a \emph{spectral measure} if there exists a countable subset $\Lambda\subset\mathbb{R}^n$ such that the family of exponential functions $E(\Lambda):=\{e^{2\pi i\langle\lambda,x\rangle}:\lambda\in\Lambda\}$ forms an orthonormal basis for $L^2(\mu)$. In this case, the set $\Lambda$ is called a \emph{spectrum} of $\mu$, and we also say that $(\mu,\Lambda)$ is a \emph{spectral pair}.
In particular, if $\mu$ is the normalized Lebesgue measure supported on a Borel set $\Omega$, then $\Omega$ is called a \emph{spectral set}.
The classical example of a spectral set is the unit cube $\Omega=[-\frac{1}{2},\frac{1}{2}]^n$, for which the set $\Lambda=\mathbb{Z}^n$ serves as a spectrum.

Spectral theory for the Lebesgue measures on sets has been studied extensively since it initialed by Fuglede \cite{Fuglede_1974}, whose famous conjecture asserted that $\Omega$ is a  spectral set on $\mathbb{R}^n$ if and only if $\Omega$ is a translational tile on $\mathbb{R}^n$. By \emph{translational tile} we mean that there exists a set $\mathcal{J }$ (called a \emph{tiling set}) such that
$$\mathbb{R}^n=\bigcup_{t\in\mathcal{J }}(\Omega+t) \quad {\rm and} \quad |(\Omega+t)\cap(\Omega+t^\prime)|=0\quad {\rm for \ all} \ t\neq t^\prime \in\mathcal{J },$$
where $|\cdot|$ denotes the Lebesgue measure.
Fuglede  proved that this conjecture holds if the spectrum  or the tiling set is assumed to be a lattice. The conjecture remains open for 30 years until Tao \cite{Tao_2004} gave the first counterexample: there exists a spectral subset of $\mathbb{R}^n$ with $n\geq5$ which is not a tile. After that, the counterexample was modified to show that the conjecture is false in both directions for $n\geq3$, see \cite{Kolountzakis-Matolcsi_2006-1,Kolountzakis-Matolcsi_2006-2}.
However, the conjecture is still open in
low dimensions $n=1,2$. As of today there has been an effort to prove or disprove Fuglede's conjecture, one can refer to \cite{B-M_2014,F-F-L-S_2019,I-K-T_2003,Laba_2001} and the references therein.

After the original work of Fuglede, the study of spectral measures is also blooming in the fractal community. Recently, He, Lai and Lau \cite{H-L-L_2013} proved that a spectral measure $\mu$ must be of pure type, that is, $\mu$ is discrete with finite support, singularly continuous or absolutely continuous with respect to Lebesgue measure. For the absolutely continuous case, a general result was proved by Dutkay and Lai \cite{Dutkay-Lai_2014} that an absolutely continuous measure is a spectral
one only if it is a normalized Lebesgue restricting on a domain $\Omega\subset\mathbb{R}^n$, which indeed involves the Fuglede's conjecture. For the singular continuous case, the first spectral measure was
constructed by Jorgensen and Pedersen in 1998 \cite{Jorgensen-Pedersen_1985}. They proved that  the set
$\Lambda=\{\sum_{k=0}^{n}4^{k}d_k:d_k\in \{0,1\},\ n\in\mathbb{N}\}$ is a spectrum of the standard middle-fourth Cantor measure.
In the same paper, they also showed that the usual middle-third Cantor measure is not spectral.  Since then, many interesting spectral measures have been found, at the same time some singular phenomena different from the spectral theory of Lebesgue measures have been discovered, please see \cite{An-He_2014,An-He-Lau_2015,An-He-Tao_2015,Dai_2012,Dai-He-Lai_2013,Dai-He-Lau_2014,Deng-Lau_2015,Deng_2016,Dutkay-Jorgensen_2007,
Dutkay-Jorgensen_2007_1,Dutkay-Haussermann-Lai_2019,Hu-Lau_2008,D-J_2006_1,Laba-Wang_2002,Liu-Luo_2017,Jorgensen-Pedersen_1985,Liu-Wang} and the references therein. Furthermore, the theory of spectral measures has been connected with many popular mathematical fields, such as the Fourier frame \cite{H-L-L_2013,N-O-U_2016}, wavelet  and Riesz basis \cite{Dutkay-Jorgensen_2006} and so on.

Consider the iterated function system (IFS) on $\mathbb{R}^n$,
$$\phi_d(x)=M^{-1}(x+d), \quad x\in \mathbb{R}^n,\ d\in D,$$
where $M\in M_n(\mathbb{R})$ is  an  expanding real matrix  (i.e., all eigenvalues of $M$  have
modulus strictly greater than one) and $D\subset\mathbb{R}^n$ is a finite digit set with cardinality $\#D$. It follows that
there exists a unique nonempty compact subset $T(M,D)\subset\mathbb{R}^n$ satisfying $T(M,D)=\bigcup_{d\in D}\phi_d(T(M,D))$ (see \cite{Hutchinson_1981} for details). Also there exists a unique probability measure $\mu_{M,D}$ supported on $T(M,D)$ such that
\begin{equation}\label{1.1}
\mu_{M,D}=\frac{1}{\#D}\sum_{d\in D}\mu_{M,D}\circ\phi_d^{-1}.
\end{equation}
We call $T(M,D)$ a \emph{self-affine set} (or \emph{attractor}) and $\mu_{M,D}$ a \emph{Sierpinski-type  self-affine measure}, respectively. In particular, if $M\in M_n(\mathbb{Z})$ and $D\subset\mathbb{Z}^n$, we call the $\mu_{M,D}$ an integral Sierpinski-type self-affine measure.

For a given measure $\mu_{M,D}$, the most interesting problem in this respect is to determine the spectrality or non-spectrality of $\mu_{M,D}$. And in the case when it is a spectral measure, one needs to determine the Fourier bases in the Hilbert space $L^2(\mu_{M,D})$. All these are directly
connected with  the Fourier transform $\hat{\mu}_{M,D}$ (see \eqref{2.1}) of the measure $\mu_{M,D}$ simply due to the fact that $\mu_{M,D}$ is a spectral measure with a spectrum $\Lambda$ if and only if $\sum_{\lambda\in\Lambda}|{\hat{\mu}(\xi+\lambda)}|^2=1$ for all $\xi\in\mathbb{R}^n$ (see \cite[Lemma 3.3]{Jorgensen-Pedersen_1985}). Moreover, the appearance of compatible
pair (known also as Hadamard triple), following the terminology of \cite{Strichartz_2000}, leads the research to an approachable way.

\begin{de}\label{deA}
Let $M\in M_n(\mathbb{Z})$ be an $n\times n$ expanding integer matrix, and let $D,S\subset\mathbb{Z}^n$ be two finite digit
sets with the same cardinality. We say
that $(M, D)$ is {\it admissible} (or $(M^{-1}D, S)$ forms a {\it compatible
pair} or $(M,D,S)$ forms a {\it Hadamard triple}) if the matrix
$$
H=\frac{1}{\sqrt{\#D}}\left[  e^{2\pi i\langle M^{-1}d, s\rangle}\right]_{d\in D, s\in S}
$$
is unitary, i.e., $H^*H=I$, where $H^*$ denotes the conjugate transposed matrix of $H$.
\end{de}

Given a discrete set $A\subset\mathbb{R}^n$, we define the discrete measure on $A$ by $\delta_A=\frac{1}{\#A}\sum_{a\in A}\delta_a$, where $\delta_a$ is the Dirac measure at the point $a$. It is easy to see that the discrete measure $\delta_{M^{-1}D}$ is a spectral measure  if and only if
$(M, D)$ is admissible. The well-known result of Jorgensen and Pedersen \cite{Jorgensen-Pedersen_1985} shows that
if $(M^{-1}D, S)$ is a compatible pair, then $E(\Lambda(M,S))$
is an infinite orthogonal system in $L^2(\mu_{M,D})$, where
\begin{equation*}
\Lambda(M,S)=\left\{\sum_{j=0}^{k-1} M^{*j}s_j:k\geq1, \ s_j\in S \right\}.
\end{equation*}
Moreover, Dutkay and Jorgensen \cite[Conjecture 2.5]{Dutkay-Jorgensen_2007_1}\cite[Conjecture 1.1]{Dutkay-Jorgensen_2009} formulated the following well-known conjecture:

\noindent\emph{{\bf Conjecture.} The integral Sierpinski-type self-affine measure $\mu_{M,D}$ is a spectral measure if $(M,D)$ is admissible.}

The conjecture implies that the existence of a set $S$ such that $(M,D,S)$ is a Hadamard triple is sufficient to obtain orthonormal bases of exponential functions  in $L^2(\mu_{M,D})$.  It was first proved on $\mathbb{R}$ by {\L}aba and Wang \cite{Laba-Wang_2002} and later refined in \cite{D-J_2006_1}.  The situation becomes more complicated when $n>1$. In the  paper \cite{Dutkay-Jorgensen_2007_1}, the authors showed that the conjecture is true if $(M,D)$ satisfies a technical condition called the \emph{reducibility condition}. The conjecture is true
under some additional assumptions, introduced by Strichartz \cite{Strichartz_2000}. Some high dimensional special cases were also considered by Li \cite{Li_2012,Li_2013}.  Eventually, Dutkay, Haussermann and Lai \cite{Dutkay-Haussermann-Lai_2019} proved that this conjecture is true, Hadamard triples always generate self-affine spectral measures.

\begin{thm} \cite[Theorem 1.3]{Dutkay-Haussermann-Lai_2019} \label{th(1.2)}
The integral Sierpinski-type self-affine measure $\mu_{M,D}$ is a spectral measure if $(M,D)$ is admissible.
\end{thm}

In particular, for the integral self-affine measure $\mu_{M,D}$ generated by an expanding matrix $M\in M_2(\mathbb{Z})$ and the digit set $D=\{(0,0)^t,(1,0)^t,(0,1)^t\}$, there are many papers studied the spectrality and non-spectrality of it \cite{Dutkay-Jorgensen_2007,Li_2008,Li_2010,Liu-Dong-Li_2017}. Finally, the spectrality of $\mu_{M,D}$ has been completely characterized by An, He and Tao \cite{An-He-Tao_2015}, who showed that $\mu_{M,D}$ is a spectral measure if and only if $(M, D)$ is admissible. For the general integer digit set
 \begin{equation}\label{1.2}
D=\left\{\begin{pmatrix}
0\\0\end{pmatrix},\begin{pmatrix}
\alpha_{1}\\ \alpha_{2}
\end{pmatrix},
\begin{pmatrix}
\beta_{1}\\ \beta_{2}
\end{pmatrix}\right\}
\end{equation}
with $\alpha_1\beta_2-\alpha_2\beta_1\neq0$, to the best of our knowledge,
Chen and Liu  \cite{Chen-Liu_2019}  first proved that if $\det(M)\notin 3\mathbb{Z}$, then the  mutually orthogonal exponential functions in $L^2(\mu_{M,D})$ is finite. Recently,  Liu and Wang \cite{Liu-Wang} further considered the case $\alpha_1\beta_2-\alpha_2\beta_1\notin 3\mathbb{Z}$, and gave the necessary and sufficient conditions for $\mu_{M,D}$ to be a spectral measure. In order to  characterize the spectral property of $\mu_{M,D}$, they introduced the general linear group. For a prime $p$, let $\mathbb{F}_p=\mathbb{Z}/p\mathbb{Z}$ be the residue class fields.
All nonsingular $n\times n$ matrices over $\mathbb{F}_p$ form a finite group under matrix multiplication, which is
called the {\it general linear group} $GL_n(p)$. The following is their main result.

\begin{thm}\cite[Theorem 1.10]{Liu-Wang}\label{th(Liu)}
Let $\mu_{M,D}$ be defined by \eqref{1.1},  where $M\in M_2(\mathbb{Z})$  is an expanding matrix and the digit set $D$ is given by \eqref{1.2} with $\alpha_1\beta_2-\alpha_2\beta_1\notin 3\mathbb{Z}$. Then the following statements  are equivalent.
\begin{enumerate}[{\rm (i)}]
  \item $\mu_{M,D}$ is a spectral measure.
 \item $(M, D)$ is admissible.
 \item $(AMB)^*(1,-1)^t\in 3\mathbb{Z}^2$,
where  $B=\begin{bmatrix}
\alpha_1&\beta_1\\
\alpha_2&\beta_2
\end{bmatrix}$ and  $A\in GL_2(3)$ satisfies  $AB=I\;({\rm{ mod} } \ M_2(3\mathbb{Z})).$
\end{enumerate}
\end{thm}

As stated above, for an expanding  matrix $M\in M_2(\mathbb{Z})$ and the digit set $D$ given by \eqref{1.2}, the spectral properties of the corresponding measure $\mu_{M,D}$ in the case of $\det(M)\notin 3\mathbb{Z}$ or $\alpha_1\beta_2-\alpha_2\beta_1\notin 3\mathbb{Z}$ have been completely characterized.
A natural subsequent question is that

{\bf(Qu 1):} For an expanding matrix $M\in M_2(\mathbb{Z})$ with $\det(M)\in 3\mathbb{Z}$ and the digit set $D$ given by \eqref{1.2} with $\alpha_1\beta_2-\alpha_2\beta_1\in 3\mathbb{Z}$, what is the sufficient and necessary condition for $\mu_{M,D}$ to be a spectral measure? Is it $(M,D)$ admissible?

In this paper, our main purpose is to initiate a study on the  question {\bf(Qu 1)}. Without loss of generality, we can assume that $\gcd(\alpha_1,\alpha_2,\beta_1,\beta_2)=1$ by Lemma \ref{th(2.1)}.
The main result in the paper is the following:

\begin{thm} \label{th(1.4)}
Let $\mu_{M,D}$ be defined by \eqref{1.1},  where $M\in M_2(\mathbb{Z})$  is an expanding matrix with $\det(M)\in 3\mathbb{Z}$ and the digit set $D$ is given by \eqref{1.2} with $\alpha_1\beta_2-\alpha_2\beta_1\in 3\mathbb{Z}$. Then $\mu_{M,D}$ is a spectral measure if and only if there exists a matrix  $Q\in M_2(\mathbb{Z})$  such that  $(M',D')$ is admissible, where $M'=QMQ^{-1}$ and $D'=QD$. In particular, if $2\alpha_1-\beta_1, 2\alpha_2-\beta_2\in3\mathbb{Z}$, then $\mu_{M,D}$ is a spectral measure if and only if $L^2(\mu_{M,D})$ contains an infinite orthogonal set of exponential functions.
\end{thm}

We remark that Theorem \ref{th(1.4)} settles completely the spectrality question {\bf(Qu 1)} for the planar self-affine measure $\mu_{M,D}$. In Theorem \ref{th(1.4)}, it is worth noting that $Q$ is determined by the matrix $M$.  That is, for different $M$, the corresponding $Q$ may be different  (see Theorem \ref{th(3.4)}).  We also notice that
$Q$ is nontrivial, i.e.,  there exist $M$ and $D$ such that $(M,D)$ is not admissible, but $(QMQ^{-1},QD)$ is admissible (see Remark \ref{th(4.10)}).

We now outline the strategy of the proof of Theorem \ref{th(1.4)}. First, we give an  equivalent form of Theorem \ref{th(1.4)} under a similarity transformation (see Theorem \ref{th(2.4)}), which implies that the spectrality  of $\mu_{M,D}$  is the same as that of $\mu_{\tilde M,\tilde D}$, where $\tilde D$ and $\tilde M$  are given by \eqref{2.6} and \eqref{2.7} respectively.
For the first result of Theorem \ref{th(1.4)}, the sufficiency will follow directly from Theorem \ref{th(1.2)} and Lemma \ref{th(2.1)}, the difficult part of the proof is the necessity.
A key point in the proof is to analyze the characteristics of the zero set of the mask polynomial of $\tilde D$ (see Proposition \ref{pro(A)}). Moreover, our approach to the result is based on how to express the matrix $\tilde M$. We will divide the proof of Theorem \ref{th(1.4)} into the following two cases:

\textbf{Case I:}  $2\alpha_1-\beta_1, 2\alpha_2-\beta_2\in3\mathbb{Z}$.

\textbf{Case II:} $2\alpha_1-\beta_1\notin 3\mathbb{Z}$ or $2\alpha_2-\beta_2\notin 3\mathbb{Z}$.

\noindent In each case, we will use different methods to characterize the spectral properties of the measure $\mu_{\tilde M,\tilde D}$ corresponding to different matrices $\tilde M$ (see Theorems \ref{th(3.4)} and \ref{th(4.7)}).

For the second result of Theorem \ref{th(1.4)}, the necessity is essentially trivial from the definition of spectral measure. To complete the proof the sufficiency, we first use the properties of  $\tilde{M}$ and $\tilde D$ to show that $\mu_{\tilde M,\tilde D}$ is a non-spectral measure for $\tilde{M}\in\mathfrak{B}$ (see Lemma \ref{th(3.2)}), and then  prove the completeness of the infinite orthogonal set of  exponential functions (see Theorem \ref{th(3.3)}), where the completeness is established by checking the existence of Hadamard triple.

The paper is organized as follows. In Section 2, we recall a few basic concepts and notations, establish several lemmas that will be needed in the proof of our main results. In Section 3, we focus on proving  Theorem \ref{th(1.4)} in Case I. Finally,   we  settle Case II, and  give some remarks and  an open problem in Section 4.

\section{\bf Preliminaries\label{sect.2}}

The purpose of this section is to collect necessary facts that we need in the following
sections. Let $M\in M_n(\mathbb{R})$ be  an  expanding real matrix, $D\subset\mathbb{R}^n$ be a finite digit set with cardinality $\#D$, and let  $\mu_{M,D}$ be defined by \eqref{1.1}. The Fourier transform of $\mu_{M,D}$ is defined as usual,
\begin{equation}\label{2.1}
\hat{\mu}_{M,D}(\xi)=\int e^{2\pi i\langle x,\xi\rangle }d\mu_{M,D}(x)=\prod_{j=1}^\infty m_D({M^{*}}^{-{j}}\xi), \quad  \xi\in\mathbb{R}^n,
\end{equation}
where $M^*$ denotes the conjugate transposed matrix of $M$, and $m_D(x)$ is the mask polynomial of $D$, which is defined by
\begin{equation*}
m_D(x)=\frac 1{\#D}\sum\limits_{d\in D}{e^{2\pi i\langle d,x\rangle}},\quad x\in\mathbb{R}^n.
\end{equation*}
Let $\mathcal{Z}(f)$ denote the zero set of the function $f$, i.e., $\mathcal{Z}(f)=\{x\in\mathbb{R}^n:f(x)=0\}$, and define
\begin{equation}\label{2.2}
\mathcal{Z}_{D}^n=\mathcal{Z}(m_{D})\cap [0, 1)^n.
\end{equation}
It is easy to see that $m_D$ is a $\mathbb{Z}^n$-periodic function for $D\subset \mathbb{Z}^n$. In this case,
\begin{equation}\label{2.3}
\mathcal{ Z}(m_D)=\mathcal{Z}_{D}^n+\mathbb{Z}^n.
\end{equation}
It follows from \eqref{2.1} that
\begin{equation}\label{2.4}
\mathcal{Z}(\hat{\mu}_{M, D})=\bigcup_{j=1}^{\infty}M^{*j}(\mathcal{Z}(m_D)).
\end{equation}
 For any $\lambda_1\neq \lambda_2\in \mathbb{R}^n$, the orthogonality condition
\begin{equation*}
0=\langle e^{2\pi i \langle \lambda_1,x\rangle},e^{2\pi i \langle \lambda_2,x\rangle}\rangle_{L^2(\mu_{M,D})}=\int e^{2\pi i \langle\lambda_1-\lambda_2,x\rangle}d\mu_{M,D}(x)=\hat{\mu}_{M,D}(\lambda_1-\lambda_2)
\end{equation*}
relates to the zero set $\mathcal{Z}(\hat{\mu}_{M,D})$ directly. For a countable subset $\Lambda\subset\mathbb{R}^n$,  it is easy to see that $E(\Lambda)=\{e^{2\pi i\langle\lambda,x\rangle}:\lambda\in\Lambda\}$ is an orthonormal family of $L^2(\mu_{M,D})$ if and only if
 \begin{equation}\label{2.5}
 (\Lambda-\Lambda)\setminus\{0\}\subset\mathcal{ Z}(\hat{\mu}_{M,D}).
\end{equation}
We call $\Lambda$ satisfying \eqref{2.5} a bi-zero set of $\mu_{M,D}$. Since the properties of bi-zero sets (or spectra) are invariant under  a translation, it will be convenient to assume that $0\in\Lambda$ in this paper, and hence $\Lambda\subset(\Lambda-\Lambda)$.

The following lemma indicates that the spectral properties of
$\mu_{M,D}$ are invariant under a similarity transformation. The proof is the same as that of Lemma 4.1 in {\rm \cite{Dutkay-Jorgensen_2007}}.

\begin{lemma}\label{th(2.1)}
Let $D, D_1\subset \Bbb R^n$ be two finite digit sets with the same cardinality,  and let $M, M_1\in M_n(\Bbb R)$ be two expanding matrices. If there exists a matrix  $A\in M_n(\mathbb{R})$  such that $M_1=AMA^{-1}$ and $D_1=AD$,  then $\Lambda$ is a bi-zero set of $\mu_{M,D}$ if and only if $A^{*-1}\Lambda$ is a bi-zero set of $\mu_{M_1,D_1}$.
Moreover,  $\mu_{M,D}$ is a spectral measure with spectrum $\Lambda$ if and only if  $\mu_{M_1,D_1}$ is a spectral measure with spectrum $A^{*-1}\Lambda$.
\end{lemma}

In \cite{Li2015}, Li gave a necessary and sufficent condition for the finite $\mu_{M,D}$-orthogonality, which will be used to prove our main results.

\begin{thm}\cite[Theorem 2.1]{Li2015}\label{th(Li)}
Let $M\in M_{n}(\mathbb{Z})$ be an expanding matrix and $D\subset \mathbb{Z}^{n}$ be a finite digit set with $0\in D$, and let $\mathcal{Z}_{D}^n$ be defined by \eqref{2.2}. Suppose that  $\mathcal{Z}_{D}^n\subset \mathbb{Q}^n $ is a finite set, then
there exist at most finite mutually orthogonal exponential functions in $L^2(\mu_{M,D})$ if and only if $M^{*j}\mathcal{Z}_{D}^n\cap \mathbb{Z}^n=\emptyset$ for all $j\in \mathbb{N}$.
\end{thm}

The following lemma is an effective method to illustrate that a countable set $\Lambda$ cannot be a spectrum of a measure $\mu$, which will be used in Proposition \ref{th(4.1)}.

\begin{lemma}\cite[Lemma 2.2]{Dai-He-Lau_2014}\label{lem(DHL)}
Let $\mu=\mu_1\ast\mu_2$ be the convolution of two probability measures $\mu_i$, $i=1, 2$,
and they are not Dirac measures.
Suppose that $\Lambda$ is  a bi-zero set of $\mu_1$,
then $\Lambda$ is also  a bi-zero set of $\mu$,
but $\Lambda$ cannot be a spectrum of $\mu$.
\end{lemma}

In order to prove Theorem \ref{th(1.4)}, we will give an  equivalent form of Theorem \ref{th(1.4)} under a similarity transformation. Before stating the form, some technical work needs to be done.

For the matrix $M$ and the digit set $D$ given in Theorem \ref{th(1.4)},  we can let $M=\begin{bmatrix}
a&b \\
c&d
\end{bmatrix}\in M_2(\mathbb{Z})$ and $\alpha_1\beta_2-\alpha_2\beta_1=3^{\eta}\gamma$ for some integers $\eta\geq1$ and  $3\nmid \gamma$. Without loss of generality, we assume $\gcd(\alpha_1,\alpha_2)=\sigma$ with $3\nmid\sigma$ (Otherwise, we can  choose $\sigma=\gcd(\beta_1,\beta_2)$ with $3\nmid\sigma$, since $\gcd(\alpha_1,\alpha_2,\beta_1,\beta_2)=1$).   Let $\alpha_1=\sigma t_1$ and $\alpha_2=\sigma t_2$ with $\gcd(t_1,t_2)=1$, then there exist two integers $p$ and $q$ such that $pt_1+qt_2=1$. Clearly, $\sigma=p\alpha_1+q\alpha_2$ and $\sigma\mid\gamma$. For convenience, we denote $\omega=p\beta_1+q\beta_2$ and $\vartheta=\gamma/\sigma\notin 3\mathbb{Z} $. Let
$P=\begin{bmatrix}
 p&q \\
 -t_2&t_1
\end{bmatrix}$. By noting that $t_2\alpha_1=t_1\alpha_2$ and $t_1\beta_2-t_2\beta_1=3^\eta\vartheta$, we have
\begin{eqnarray}\label{2.6}
\tilde{D}=PD=\begin{bmatrix}
 p&q \\
 -t_2&t_1
\end{bmatrix}\left\{\begin{pmatrix}
0\\0\end{pmatrix},\begin{pmatrix}
\alpha_{1}\\ \alpha_{2}
\end{pmatrix},
\begin{pmatrix}
\beta_{1}\\ \beta_{2}
\end{pmatrix}\right\}=\left\{\begin{pmatrix}
0\\0\end{pmatrix},\begin{pmatrix}
\sigma\\ 0
\end{pmatrix},
\begin{pmatrix}
\omega\\ 3^{\eta}\vartheta
\end{pmatrix}\right\}\subset\mathbb{Z}^2
\end{eqnarray}
and
\begin{eqnarray}\label{2.7}
\tilde{M}=PMP^{-1}
=\begin{bmatrix}
(pa+qc)t_1+(pb+qd)t_2 & (pb+qd)p-(pa+qc)q \\
(ct_1-at_2)t_1+(dt_1-bt_2)t_2 & (dt_1-bt_2)p-(ct_1-at_2)q
\end{bmatrix}
\end{eqnarray}
It is easy to verify that $\tilde M$ is an expanding integer matrix with $\det(\tilde M)=\det(M)\in 3\mathbb{Z}$. We remark that  $2 \sigma-\omega \in3\mathbb{Z}$ if $2\alpha_1-\beta_1, 2\alpha_2-\beta_2\in3\mathbb{Z}$, and
$2 \sigma-\omega \notin3\mathbb{Z}$ if $2\alpha_1-\beta_1\notin 3\mathbb{Z}$ or $2\alpha_2-\beta_2\notin 3\mathbb{Z}$ (see Proposition \ref{pro(CL)}).

It is known that the spectrality  of  $\mu_{\tilde M,\tilde D}$ is the same as that of $\mu_{M,D}$ by Lemma \ref{th(2.1)}. Hence the proof of Theorem \ref{th(1.4)} is equivalent to proving the following

\begin{thm} \label{th(2.4)}
Let $\mu_{\tilde M,\tilde D}$ be defined by \eqref{1.1},  where $\tilde D$ and $\tilde M$  are given by \eqref{2.6} and \eqref{2.7} respectively. Then $\mu_{\tilde M,\tilde D}$ is a spectral measure if and only if there exists a matrix  $Q\in M_2(\mathbb{Z})$  such that  $(\bar{M},\bar{D})$ is admissible, where $\bar{M}=Q\tilde{M}Q^{-1}$ and $\bar{D}=Q\tilde{D}$. In particular, if $2 \sigma-\omega \in3\mathbb{Z}$, then $\mu_{\tilde M,\tilde D}$ is a spectral measure if and only if $L^2(\mu_{\tilde M,\tilde D})$ contains an infinite orthogonal set of exponential functions.
\end{thm}

Now we introduce some properties of the digit set $\tilde D$ given in Theorem \ref{th(2.4)}. It is known that
$1+e^{2\pi ix_1}+e^{2\pi ix_2}=0$
if and only if
\begin{equation*}
 \left\{
 \begin{array}{ll}
x_1=\frac{1}{3}+k_1,
\;\;\\
x_2=\frac{2}{3}+k_2,
\end{array}
\right.
\quad {\rm or} \quad
\left\{
 \begin{array}{ll}
x_1=\frac{2}{3}+k_3,
\;\;\\
x_2=\frac{1}{3}+k_4,
\end{array}
\right.
\end{equation*}
where~$k_1,k_2,k_3,k_4\in \mathbb{Z}$. By a direct calculation, we have that
\begin{eqnarray}\label{2.8}
\mathcal{Z}(m_{\tilde D})=Z_0 \cup \tilde{Z}_0,
\end{eqnarray}
where
\begin{equation*}
Z_0=\left\{
\begin{pmatrix}
  \frac{1+3k_1}{3\sigma} \\ \frac{1}{3^{\eta+1}\gamma}(2\sigma-\omega-3\omega k_1+3\sigma k_2)
  \end{pmatrix}\
  :k_1,k_2\in\mathbb{Z}
  \right\},
 \end{equation*}
and
\begin{equation*}
\tilde{Z}_0=\left\{
\begin{pmatrix}
  \frac{2+3k_3}{3\sigma} \\ \frac{1}{3^{\eta+1}\gamma}(\sigma-2\omega-3\omega k_3+3\sigma k_4)
  \end{pmatrix}\
  :k_3,k_4\in\mathbb{Z}
  \right\}.
 \end{equation*}
The following result was proved in \cite[Proposition 3.8]{Chen-Liu_2019} and illustrates some  inner relationships between two digit sets $D$ and $\tilde D$. Since the proof is simple, we give it here.
\begin{pro}\label{pro(CL)}
Let $D$ and $\tilde D$ be defined by \eqref{1.2} and \eqref{2.6}, respectively. Then
\begin{enumerate}[{\rm (i)}]
  \item $2 \sigma-\omega \in3\mathbb{Z}$ if $2\alpha_1-\beta_1, 2\alpha_2-\beta_2\in3\mathbb{Z}$.
 \item $2 \sigma-\omega \notin3\mathbb{Z}$ if $2\alpha_1-\beta_1\notin 3\mathbb{Z}$ or $2\alpha_2-\beta_2\notin 3\mathbb{Z}$.
\end{enumerate}
\end{pro}
\begin{proof}
It is easy to verify that
$2\sigma-\omega=p(2\alpha_1-\beta_1)+q(2\alpha_2-\beta_2)$, thus (i) holds.
For (ii), without loss of generality, we assume $2\alpha_1-\beta_1\notin 3\mathbb{Z}$. Since $\alpha_1=\sigma t_1$, $\alpha_2=\sigma t_2$ and $t_1\beta_2-t_2\beta_1=3^\eta\vartheta$, we have $t_2\alpha_1=t_1\alpha_2$ and
\begin{align}\label{eq(2.9)}
q3^\eta\vartheta=q(t_1\beta_2-t_2\beta_1)=qt_1(\beta_2-2\alpha_2)+qt_2(2\alpha_1-\beta_1).
\end{align}
Multiplying $2\alpha_1-\beta_1$ on both sides of $pt_1+qt_2=1$, we get $pt_1(2\alpha_1-\beta_1)+qt_2(2\alpha_1-\beta_1)=2\alpha_1-\beta_1$. This together with \eqref{eq(2.9)} yields that $t_1(p(2\alpha_1-\beta_1)+q(2\alpha_2-\beta_2))=-q3^\eta\vartheta+2\alpha_1-\beta_1$. Hence $2 \sigma-\omega=p(2\alpha_1-\beta_1)+q(2\alpha_2-\beta_2)\notin 3\mathbb{Z}$ by $2\alpha_1-\beta_1\notin 3\mathbb{Z}$.
\end{proof}

In order to characterize the properties of $\mathcal{Z}(m_{\tilde D})$, we first define the following four sets. Let
\begin{equation*}
\mathcal{H}=\left\{
\begin{pmatrix}
  \frac{\ell_{1}}{3\gamma} \\ \frac{\ell_{2}}{3^{\eta}\gamma}
  \end{pmatrix}
  : \ell_{1}\in\mathbb{Z}\setminus3\mathbb{Z},
  \ell_{2}\in\mathbb{Z}
  \right\}, \quad
\mathcal{G}=\left\{
\begin{pmatrix}
  \frac{\ell_{1}}{3\gamma} \\ \frac{\ell_2}{3^{\eta+1}\gamma}
  \end{pmatrix}
  : \ell_{1},\ell_{2}\in\mathbb{Z}\setminus3\mathbb{Z}
  \right\},
\end{equation*}
\begin{equation*}
\mathcal{G}_{1}=\left\{
\begin{pmatrix}
  \frac{\ell_1}{3} \\ \frac{\ell_2}{3^{\eta+1}}
  \end{pmatrix}
  :\ell_{1}, \ell_{2}\in\mathbb{Z}\setminus3\mathbb{Z}, \ell_{1}=\ell_{2}\;({\rm mod} \ 3)
  \right\}
\end{equation*}
and
\begin{equation*}
\mathcal{G}_{2}=\left\{
\begin{pmatrix}
  \frac{\ell_1}{3} \\ \frac{\ell_2}{3^{\eta+1}}
  \end{pmatrix}
  :\ell_{1}, \ell_{2}\in\mathbb{Z}\setminus3\mathbb{Z}, \ell_{1}\neq\ell_{2}\;({\rm mod} \ 3)
  \right\}.
\end{equation*}
\noindent
The following is an elementary but useful fact in our investigation.
\begin{pro}\label{pro(A)}
With the above notations, we have the following assertions:
\begin{enumerate}[{\rm (i)}]
  \item If $2 \sigma-\omega \in3\mathbb{Z}$, then $\gamma\mathcal{H}\subset\mathcal{Z}(m_{\tilde D})\subset\mathcal{H}$.
 \item If $2 \sigma-\omega \notin3\mathbb{Z}$, then $\mathcal{G}_{i}\subset\mathcal{Z}(m_{\tilde D})\subset\mathcal{G}$ for $i=1$ or $2$.
\end{enumerate}
\end{pro}

\begin{proof}
(i) Since $2 \sigma-\omega \in3\mathbb{Z}$ and $\gamma=\sigma\vartheta$, it follows from \eqref{2.8} that $\mathcal{Z}(m_{\tilde D})\subset\mathcal{H}$. Now we prove $\gamma\mathcal{H}\subset\mathcal{Z}(m_{\tilde D})$. It is equivalent to proving that for any $(\frac{\ell_1}{3}, \frac{\ell_2}{3^{\eta}})^t\in\gamma\mathcal{H}$ with $\ell_{1}\notin3\mathbb{Z}$ and  $\ell_{2}\in\mathbb{Z}$, there exist $k_1,k_2\in \mathbb{Z}$ such that
\begin{align}\label{2.10}
\left\{
\begin{array}{ll}
1+3k_1=\sigma\ell_1, \\
2\sigma -\omega-3\omega k_1+3\sigma k_2=3\gamma\ell_2,
\end{array}
\right. \quad {\rm or} \quad
\left\{
\begin{array}{ll}
2+3k_1=\sigma\ell_1, \\
\sigma -2\omega-3\omega k_1+3\sigma k_2=3\gamma\ell_2.
\end{array}
\right.
\end{align}
Notice that $\sigma,\ell_{1}\notin3\mathbb{Z}$, then there must exist $k_1\in \mathbb{Z}$  such that $\sigma\ell_{1}=3k_1+1$ or $\sigma\ell_{1}=3k_1+2$. Without loss of generality, we assume that $\sigma\ell_{1}=3k_1+1$, which yields $\sigma=\ell_{1}\pmod 3$. In view of $\gamma=\sigma\vartheta$ and \eqref{2.10}, we only need to prove  $$k_2=\vartheta\ell_2+\frac{\omega\ell_1-2}{3}\in\mathbb{Z}.$$
According to  $\sigma\notin3\mathbb{Z}$, $2 \sigma-\omega \in3\mathbb{Z}$ and $\sigma=\ell_{1} \pmod 3$, we conclude $\omega\notin3\mathbb{Z}$ and $\omega\neq\ell_{1} \pmod 3$. This implies that
$\omega\ell_1=3a+2$ for some $a\in\mathbb{Z}$, thus $k_2=\vartheta\ell_2+a\in \mathbb{Z}$. Hence \eqref{2.10} holds. This proves $\gamma\mathcal{H}\subset\mathcal{Z}(m_{\tilde D})$.

(ii) As $2 \sigma-\omega \notin3\mathbb{Z}$ and $\gamma=\sigma\vartheta$, it follows from \eqref{2.8} that $\mathcal{Z}(m_{\tilde D})\subset\mathcal{G}$. Now we prove $\mathcal{G}_{i}\subset\mathcal{Z}(m_{\tilde D})$ for $i=1$ or $2$. It suffices to  prove that for any $(\frac{\ell_1}{3}, \frac{\ell_2}{3^{\eta+1}})^t\in\mathcal{G}_{i}$ with $\ell_{1}, \ell_{2}\notin3\mathbb{Z}$, $\ell_{1}=\ell_{2}\;({\rm mod} \ 3)$ or $\ell_{1}\neq\ell_{2}\;({\rm mod} \ 3)$, there exist $k_1,k_2\in \mathbb{Z}$ such that
\begin{align}\label{2.11}
\left\{
\begin{array}{ll}
1+3k_1=\sigma\ell_1, \\
2\sigma -\omega-3\omega k_1+3\sigma k_2=\gamma\ell_2,
\end{array}
\right. \quad {\rm or} \quad
\left\{
\begin{array}{ll}
2+3k_1=\sigma\ell_1, \\
\sigma -2\omega-3\omega k_1+3\sigma k_2=\gamma\ell_2.
\end{array}
\right.
\end{align}
It follows from $\sigma,\ell_{1}\notin3\mathbb{Z}$  that $\sigma\ell_{1}=3k_1+1$ or $\sigma\ell_{1}=3k_1+2$ for some $k_1\in \mathbb{Z}$. Without loss of generality, we assume that $\sigma\ell_{1}=3k_1+1$, which means that $\ell_{1}=\sigma\;({\rm mod} \ 3)$. Using $\gamma=\sigma\vartheta$ and \eqref{2.11}, we only need to prove
\begin{equation*}
k_2=\frac{\vartheta\ell_2+\omega\ell_1-2}{3}\in\mathbb{Z}
\quad {\rm for} \ \ell_{2}=\ell_{1}\;({\rm mod} \ 3)\  {\rm or}\ \ell_{2}\neq\ell_{1}\;({\rm mod} \ 3).
\end{equation*}
We divide the proof into the following two cases.

If $\omega \in3\mathbb{Z}$, one can write $\omega=3\omega_1$ with $\omega_1 \in\mathbb{Z}$. Since $\vartheta\notin3\mathbb{Z}$ and $\ell_{1}=\sigma\;({\rm mod} \ 3)$, there  exists $b\in\mathbb{Z}$ so that $\vartheta\ell_2=3b+2$ for $\ell_{2}=\ell_{1}\;({\rm mod} \ 3)$ or $\ell_{2}\neq\ell_{1}\;({\rm mod} \ 3)$. Thus $k_2=\omega_1\ell_1+b\in \mathbb{Z}$.

If $\omega \notin3\mathbb{Z}$, then $\omega=\sigma \pmod 3$ by
$2\sigma-\omega\notin3\mathbb{Z}$. This together with $\sigma=\ell_{1}\;({\rm mod} \ 3)$ yields that $\omega\ell_1=3c+1$ for some $c\in\mathbb{Z}$.  Note that $\vartheta\notin3\mathbb{Z}$, then there must exist $d\in\mathbb{Z}$ so that $\vartheta\ell_2=3d+1$ for $\ell_{2}=\ell_{1}\;({\rm mod} \ 3)$ or $\ell_{2}\neq\ell_{1}\;({\rm mod} \ 3)$. Therefore, $k_2=c+d\in \mathbb{Z}$.

This proves \eqref{2.11}, and hence $\mathcal{G}_{i}\subset\mathcal{Z}(m_{\tilde D})$ for $i=1$ or $2$.
\end{proof}

The following lemma is an effective method to  judge whether $(\tilde M, \tilde D)$ is admissible for some special forms of $\tilde M$ and $\tilde D$. It is needed in the proof of  Theorem \ref{th(2.4)}.
\begin{lemma}\label{th(2.7)}
Let $0\in\mathcal{J}$ be a  set with $\#\mathcal{J}=3$. If  $(\mathcal{J}-\mathcal{J})\setminus\{0\}\subset\mathcal{Z}(m_{\tilde D})$ and $\tilde{M}^*\mathcal{J}\subset \mathbb{Z}^2$, then $(\tilde M, \tilde D)$ is  admissible.
\end{lemma}
\begin{proof}
Let $S=\tilde{M}^*\mathcal{J}$. Since $0\in\mathcal{J}$ and $\#\mathcal{J}=\#\tilde{D}=3$, by Definition \ref{deA},  it is clear that  $(\tilde{M}^{-1}\tilde{D},S)$ is a compatible
pair if and only if $(\mathcal{J}-\mathcal{J})\setminus\{0\}\subset\mathcal{Z}(m_{\tilde D})$. Hence the lemma follows.
\end{proof}

At the end of this section, we give two remarks that will be used to prove our main results. Throughout the paper, we denote
\begin{equation}\label{2.12}
\mathcal{J}_0=\left\{
\begin{pmatrix}
  0 \\ 0
  \end{pmatrix},
  \begin{pmatrix}
  \frac{1}{3} \\ 0
  \end{pmatrix},
 \begin{pmatrix}
  \frac{2}{3} \\ 0
  \end{pmatrix}
  \right\}, \quad
\mathcal{J}_1=\left\{
\begin{pmatrix}
  0 \\ 0
  \end{pmatrix},
  \begin{pmatrix}
  \frac{1}{3} \\ \frac{2}{3}
  \end{pmatrix},
 \begin{pmatrix}
  \frac{2}{3} \\ \frac{1}{3}
  \end{pmatrix}
  \right\} \quad {\rm and}\quad
\mathcal{J}_2=\left\{
\begin{pmatrix}
  0 \\ 0
  \end{pmatrix},
  \begin{pmatrix}
  \frac{1}{3} \\ \frac{1}{3}
  \end{pmatrix},
 \begin{pmatrix}
  \frac{2}{3} \\ \frac{2}{3}
  \end{pmatrix}
  \right\}.
\end{equation}
\begin{re}\label{re(A)}
{\rm  By \eqref{2.2} and Proposition \ref{pro(A)}, it is easy to verify that
$\mathcal{Z}_{\tilde D}^2\subset\mathbb{Q}^2$
is a finite set.
Moreover, if $2 \sigma-\omega \in3\mathbb{Z}$ and  $\eta>0$, then
$\mathcal{J}_{i}\setminus\{0\},(\mathcal{J}_{i}-\mathcal{J}_{i})\setminus\{0\}\subset\mathcal{Z}(m_{\tilde D})$ for $i\in\{0,1,2\}$.}
\end{re}

\begin{re}\label{re(B)}
{\rm If $\eta=0$ in $\tilde D$, using the same proof as  Proposition \ref{pro(A)}, the following statements hold.
\begin{enumerate}[{\rm (i)}]
  \item If $2 \sigma-\omega \in3\mathbb{Z}$, then
$\mathcal{J}_{0}\setminus\{0\},(\mathcal{J}_{0}-\mathcal{J}_{0})\setminus\{0\}\subset\mathcal{Z}(m_{\tilde D})$.
 \item If $2 \sigma-\omega \notin3\mathbb{Z}$, then
$\mathcal{J}_{i}\setminus\{0\},(\mathcal{J}_{i}-\mathcal{J}_{i})\setminus\{0\}\subset\mathcal{Z}(m_{\tilde D})$ for $i=1$ or $2$.
\end{enumerate} }
\end{re}

\section{\bf Proof of Theorem \ref{th(1.4)} in Case I\label{sect.3}}

In the present section, we focus on proving Theorem \ref{th(1.4)} in Case I, which is equivalent to proving Theorem \ref{th(2.4)} in the case of  $2 \sigma-\omega \in3\mathbb{Z}$.
To this end, we  use the residue system of modulo $3$ and rewrite the matrix $\tilde{M}$ given by \eqref{2.7} in the following form:
\begin{equation}\label{3.1}
\tilde{M}=3\begin{bmatrix}
a&b\\
3^{s-1}c&d
\end{bmatrix}+M_k:=M^\prime+M_k,
\end{equation}
where $s\geq1$, $a,b,d\in \mathbb{Z}$ and  $c\in(\mathbb{Z}\setminus3\mathbb{Z})\cup\{0\}$, and the entries of the
matrix $M_k$ are from the set $\{0,1,2\}$. It is obvious that  $s$ can be any positive integer if $c=0$. Without loss of generality, in the
rest of this paper, we always assume that $s\geq\eta$ in this case.  As $\det(\tilde{M})\in 3\mathbb{Z}$,
there are $10$ different matrices $M_k$ as following:
\begin{equation*}
\begin{array}{l}
{M_1} =\begin{bmatrix}
0&0\\0&0
\end{bmatrix},\;  \  {M_2} =\begin{bmatrix}
p_1&0\\0&0
\end{bmatrix}, \;  \  {M_3} =\begin{bmatrix}
0&p_2\\0&0
\end{bmatrix}, \;   \  {M_4} =\begin{bmatrix}
0&0\\p_3&0
\end{bmatrix}, \;   \
{M_5} =\begin{bmatrix}
0&0\\0&p_4
\end{bmatrix},\\  {M_6} =\begin{bmatrix}
p_1&0\\p_3&0
\end{bmatrix}, \;  \     {M_7} =\begin{bmatrix}
p_1&p_2\\0&0
\end{bmatrix}, \;  \ {M_8} =\begin{bmatrix}
0&0\\p_3&p_4
\end{bmatrix}, \;   \
{M_9} =\begin{bmatrix}
0&p_2\\0&p_4
\end{bmatrix}, \;  \  {M_{10}} =\begin{bmatrix}
p_1&p_2\\p_3&p_4
\end{bmatrix},
\end{array}
\end{equation*}
where $p_1,p_2,p_3,p_4\in\{1,2\}$ and $p_1p_4-p_2p_3\in 3\mathbb{Z}$. Fix $k\in\{1,2,\ldots,10\}$, we denote
\begin{equation}\label{3.2}
\mathfrak{M}_k=\left\{\tilde{M}: \tilde{M}=
M^\prime+M_k\right\}.
\end{equation}

 The following lemma illustrates a basic property of the matrix $\tilde{M}\in\mathfrak{M}_2$, which will be needed in proving Lemma \ref{th(3.2)}.
\begin{lemma}\label{th(3.1)}
If $\tilde{M}\in\mathfrak{M}_2$, then for any $\ell\in \mathbb{N}$, there exist $a_\ell,b_\ell,d_\ell\in\mathbb{Z}$ and $c_\ell\in(\mathbb{Z}\setminus3\mathbb{Z})\cup\{0\}$ such that $$\tilde{M}^{*\ell}=\begin{bmatrix}
3a_\ell+p_1^\ell&3^{s}c_\ell\\
3b_\ell&3d_\ell
\end{bmatrix}.$$
In particular, $c_\ell=0$ if $c=0$.
\end{lemma}
\begin{proof}
We prove the lemma  by induction.
It is obvious that  the lemma holds for $\ell=1$.
Inductively, we assume that it holds for $\ell=k$. That is,  $$\tilde{M}^{*k}=\begin{bmatrix}
3a_k+p_1^k&3^{s}c_k\\
3b_k&3d_k
\end{bmatrix},$$
where $a_k,b_k,d_k\in\mathbb{Z}$ and $c_k\in(\mathbb{Z}\setminus3\mathbb{Z})\cup\{0\}$. In particular, $c_k=0$ if $c=0$.

We then consider $\ell=k+1$. By inductive hypothesis, we have
\begin{eqnarray*} \nonumber
\tilde{M}^{*k+1}
&=&\begin{bmatrix}
3a_k+p_1^k&3^{s}c_k\\
3b_k&3d_k
\end{bmatrix}\begin{bmatrix}
3a+p_1&3^{s}c\\
3b&3d
\end{bmatrix}\\ \nonumber
&=&\begin{bmatrix}
3(3^sbc_k+3a a_k+ap_1^k+a_kp_1)+p_1^{k+1}&
3^s(3ca_k+3dc_k+cp_1^k)\\
3(3ab_k+3bd_k+b_kp_1)&3(3^scb_k+3d d_k)
\end{bmatrix}\\ \nonumber
&:=&\begin{bmatrix}
3a_{k+1}+p_1^{k+1}&3^{s}c_{k+1}\\
3b_{k+1}&3d_{k+1}
\end{bmatrix}.
\end{eqnarray*}
It is easy to see that $a_{k+1},b_{k+1},d_{k+1}\in\mathbb{Z}$. Since $c,c_k\in(\mathbb{Z}\setminus3\mathbb{Z})\cup\{0\}$, $p_1\in\{1,2\}$ and $c_k=0$ if $c=0$, we conclude that $c_{k+1}\in(\mathbb{Z}\setminus3\mathbb{Z})\cup\{0\}$.  In particular, if $c=0$, we have  $c_{k+1}=0$ by $c_k=0$.
This means that the lemma  holds for $\ell=k+1$.  Hence the proof is completed.
\end{proof}

For the set $\mathfrak{M}_k$ given by \eqref{3.2}, we  denote
\begin{equation}\label{3.3}
\mathfrak{B}=\left\{\tilde{M}: \tilde{M}\in\mathfrak{M}_k,\ k\in\{2,7\}, \ s\geq\eta\right\}.
\end{equation}
It should be noted that the set $\mathfrak{B}$ contains the case $c=0$ in $\tilde{M}$. Under the assumption of $2 \sigma-\omega \in3\mathbb{Z}$, we will prove that $\mu_{\tilde M,\tilde D}$ is a  non-spectral measure for any $\tilde{M}\in\mathfrak{B}$.

\begin{lemma}\label{th(3.2)}
Let $\mathfrak{B}$ be given by  \eqref{3.3}. If $\tilde{M}\in\mathfrak{B}$ and $2 \sigma-\omega \in3\mathbb{Z}$ in $\tilde D$, then there exist at most finite mutually orthogonal exponential functions in $L^2(\mu_{\tilde M,\tilde D})$.
\end{lemma}
\begin{proof}
First, we  consider $\tilde{M}\in\mathfrak{B}\cap\mathfrak{M}_2$. We claim that $\tilde{M}^{*\ell}\mathcal{Z}_{\tilde D}^2\cap \mathbb{Z}^2=\emptyset$ for any $\ell\in \mathbb{N}$. If not,  there exists  $m\in \mathbb{N}$ such that $\tilde{M}^{*m}\mathcal{Z}_{\tilde D}^2\cap \mathbb{Z}^2\neq\emptyset$. Since $2 \sigma-\omega \in3\mathbb{Z}$, it follows from Proposition \ref{pro(A)} that
$$\mathcal{Z}_{\tilde D}^2\subset \mathcal{H}\cap [0, 1)^2=\left\{
\begin{pmatrix}
  \frac{\ell_{1}}{3\gamma} \\ \frac{\ell_{2}}{3^{\eta}\gamma}
  \end{pmatrix}
  :1\leq\ell_{1}\leq 3\gamma-1, 0\leq \ell_{2}\leq 3^{\eta}\gamma-1, \ell_{1}\in\mathbb{Z}\setminus3\mathbb{Z}
  \right\}.$$
Thus there exist $(\frac{\ell_{1}}{3\gamma}, \frac{\ell_{2}}{3^{\eta}\gamma})^t\in\mathcal{H}\cap [0, 1)^2$ and $(s_1,s_2)^t\in\mathbb{Z}^2$ such that
\begin{equation}\label{3.4}
\tilde{M}^{*m}\begin{pmatrix}
  \frac{\ell_{1}}{3\gamma} \\ \frac{\ell_{2}}{3^{\eta}\gamma}
  \end{pmatrix}=\begin{pmatrix}
  s_{1} \\ s_{2}
  \end{pmatrix}.
\end{equation}
By Lemma \ref{th(3.1)},
there exist $a_m,b_m,d_m\in\mathbb{Z}$ and $c_m\in(\mathbb{Z}\setminus3\mathbb{Z})\cup\{0\}$ such that
$$\tilde{M}^{*m}=\begin{bmatrix}
3a_m+p_1^m&3^{s}c_m\\
3b_m&3d_m
\end{bmatrix}.$$
Then multiplying both sides of equation \eqref{3.4} by $3^{\eta}\gamma$, we get
\begin{equation}\label{3.5}
 \ell_{1} p_1^m=3(\gamma s_{1}-a_m\ell_{1}-3^{s-\eta}c_m\ell_{2})\quad {\rm and} \quad
 d_m\ell_{2}=3^{\eta-1}(\gamma s_{2}-b_m\ell_{1}).
\end{equation}
Since $\ell_{1}\notin3\mathbb{Z}$ and $p_1\in\{1,2\}$, on can get $\ell_{1} p_1^m\notin3\mathbb{Z}$. However, it follows from $s\geq\eta$ that $3(\gamma s_{1}-a_m\ell_{1}-3^{s-\eta}c_m\ell_{2})\in3\mathbb{Z}$. Thus \eqref{3.5} does not hold, and hence  the claim follows.
According to the claim and Theorem \ref{th(Li)}, there exist at most finite mutually orthogonal exponential functions in $L^2(\mu_{\tilde{M},\tilde{D}})$.

Second, we  consider $\tilde{M}\in\mathfrak{B}\cap\mathfrak{M}_7$. It follows from \eqref{3.2} and \eqref{3.3} that
\begin{equation*}
\tilde{M}=\begin{bmatrix}
3a+p_1&3b+p_2\\
3^{s}c&3d
\end{bmatrix},
\end{equation*}
where $s\geq\eta$, $a,b,d\in\mathbb{Z}$,  $c\in(\mathbb{Z}\setminus3\mathbb{Z})\cup\{0\}$ and $p_1,p_2\in\{1,2\}$. It is clear that there exists $\tau\in\{1,2\}$ such that $p_2-\tau p_1\in3\mathbb{Z}$. Let $Q=\begin{bmatrix}
 1&\tau \\
 0&1
\end{bmatrix}$, then we have
\begin{eqnarray*}
\tilde{M}^\prime=Q\tilde{M}Q^{-1}
=\begin{bmatrix}
3(3^{s-1}\tau c+a)+p_1&3(-3^{s-1}c\tau^2 -(a-d)\tau +b)+p_2-\tau p_1\\
3^{s}c&3(d-3^{s-1}\tau c)
\end{bmatrix}
\end{eqnarray*}
and
\begin{eqnarray*} \tilde{D}^\prime=Q\tilde{D}=\begin{bmatrix}
1&\tau \\
 0&1
\end{bmatrix}\left\{\begin{pmatrix}
0\\0\end{pmatrix},\begin{pmatrix}
\sigma\\ 0
\end{pmatrix},
\begin{pmatrix}
\omega\\ 3^{\eta}\vartheta
\end{pmatrix}\right\}=\left\{\begin{pmatrix}
0\\0\end{pmatrix},\begin{pmatrix}
\sigma\\ 0
\end{pmatrix},
\begin{pmatrix}
\omega+3^{\eta}\tau\vartheta\\ 3^{\eta}\vartheta
\end{pmatrix}\right\}\subset\mathbb{Z}^2.
\end{eqnarray*}
Combining $p_1\in\{1,2\}$ with
$p_2-\tau p_1\in3\mathbb{Z}$, we can easily know that   $\tilde{M}^\prime\in\mathfrak{M}_2\cap\mathfrak{B}$. Moreover, it follows from $2 \sigma-\omega \in3\mathbb{Z}$ and $\eta\geq1$ that $2 \sigma-(\omega+3^{\eta}\tau\vartheta) \in3\mathbb{Z}$, thus $\tilde{D}^\prime$ has the same properties as $\tilde{D}$. Therefore, using the result of $\tilde{M}\in\mathfrak{B}\cap\mathfrak{M}_2$ and Lemma \ref{th(2.1)}, one can conclude that  there exist at most finite mutually orthogonal exponential functions in $L^2(\mu_{\tilde{M},\tilde{D}})$ for all $\tilde{M}\in\mathfrak{B}\cap\mathfrak{M}_7$.

Hence we complete the proof of Lemma \ref{th(3.2)}.
\end{proof}

Recall that $\mathcal{J}_i~(i\in\{0,1,2\})$ are given by \eqref{2.12}, and  let
\begin{equation}\label{3.6}
Q_n=\begin{bmatrix}
 1&0 \\
 0&\frac{1}{3^n}
\end{bmatrix} \quad {\rm for \ any} \ n\in\mathbb{N}.
\end{equation}
They will be used many times in the rest of this paper.

Now we begin to prove Theorem \ref{th(2.4)} in the case of  $2 \sigma-\omega \in3\mathbb{Z}$,
which is equivalent to proving Theorems \ref{th(3.3)} and  \ref{th(3.4)}.

\begin{thm} \label{th(3.3)}
If $2 \sigma-\omega \in3\mathbb{Z}$ in $\tilde D$, then $\mu_{\tilde M,\tilde D}$ is a spectral measure if and only if $L^2(\mu_{\tilde M,\tilde D})$ contains an infinite orthogonal set of exponential functions.
\end{thm}
\begin{proof}
According to the definition of spectral measure, the  necessity is obvious. We now prove the sufficiency. Suppose that $L^2(\mu_{\tilde M,\tilde D})$ contains an infinite orthogonal set of exponential functions, then Lemma \ref{th(3.2)} and \eqref{3.2} imply that  $\tilde{M}\in\mathfrak{B}_1\cup\mathfrak{B}_2$,  where
\begin{equation}\label{3.7}
\mathfrak{B}_1=\left\{\tilde{M}: \tilde{M}\in\mathfrak{M}_k,\ k\in\{1,3,4,5,6,8,9,10\}\right\},
\end{equation}
and
\begin{equation}\label{3.8}
\mathfrak{B}_2=\left\{\tilde{M}: \tilde{M}\in\mathfrak{M}_k,\ k\in\{2,7\}, \ s<\eta\right\}.
\end{equation}
The proof will be divided into the following three cases.

\textbf{Case 1:} $\tilde{M}\in\mathfrak{B}_1\setminus\mathfrak{M}_3$. We first claim that $\tilde{M}^*\mathcal{J}_\kappa\subset \mathbb{Z}^2$ for some $\kappa\in\{0,1,2\}$.
If $\tilde{M}\in\{\tilde{M}: \tilde{M}\in\mathfrak{M}_k,k\in\{1,4,5,8\}\}$,
 it follows from \eqref{3.2} that $\tilde{M}$
can be written as
 $$ \tilde{M}=\begin{bmatrix}
3a&3b\\
c'&d'
\end{bmatrix},$$
where $a,b,c',d'\in \mathbb{Z}$.
By a simple calculation, we obtain that
\begin{equation}\label{3.9}
\tilde{M}^*\mathcal{J}_0=\left\{\begin{pmatrix}
0\\0\end{pmatrix},\begin{pmatrix}
a\\ b
\end{pmatrix},
\begin{pmatrix}
2a\\ 2b
\end{pmatrix}\right\}\subset \mathbb{Z}^2.
\end{equation}
If $\tilde{M}\in\mathfrak{M}_6$,
we denote
$$\tilde{M}=\tilde{M}_{6,1} \quad {\rm if} \ p_1=p_3 \quad {\rm and} \quad  \tilde{M}=\tilde{M}_{6,2} \quad {\rm if} \ p_1\neq p_3.$$
If $\tilde{M}\in\mathfrak{M}_9$,
we denote
$$\tilde{M}=\tilde{M}_{9,1} \quad {\rm if} \ p_2=p_4 \quad {\rm and} \quad \tilde{M}=\tilde{M}_{9,2} \quad {\rm if} \ p_2\neq p_4.
$$
Then by using \eqref{2.12} and \eqref{3.2}, it is easy to verify that
\begin{equation}\label{3.10}
 \tilde{M}_{k,i}^*\mathcal{J}_{i}\subset \mathbb{Z}^2\quad {\rm for} \
 k\in\{6,9\}\ {\rm and} \ i\in\{1,2\}.
\end{equation}
If $\tilde{M}\in\mathfrak{M}_{10}$,
we denote
$$\tilde{M}=\tilde{M}_{10,1} \quad {\rm if} \ p_1=p_3 \quad {\rm and} \quad  \tilde{M}=\tilde{M}_{10,2} \quad {\rm if} \ p_1\neq p_3.$$
Note that  $p_1,p_2,p_3,p_4\in\{1,2\}$ and  $p_1p_4-p_2p_3\in 3\mathbb{Z}$,
then  $p_2=p_4$ if $p_1=p_3$, and  $p_2\neq p_4$ if $p_1\neq p_3$.
Hence, it  follows from \eqref{2.12} and \eqref{3.2} that
$\tilde{M}_{10,i}^*\mathcal{J}_{i}\subset \mathbb{Z}^2$ for $i\in\{1,2\}$.
Combining this with \eqref{3.9} and \eqref{3.10}, the claim follows.

Since $2 \sigma-\omega \in3\mathbb{Z}$, it follows from Remark \ref{re(A)} that
$\mathcal{J}_{i}\setminus\{0\},(\mathcal{J}_{i}-\mathcal{J}_{i})\setminus\{0\}\subset\mathcal{Z}(m_{\tilde
D})$ for $i\in\{0,1,2\}$. Together  with the claim and                                                                                                                                                                                                                                                                                                                               Lemma \ref{th(2.7)}, it shows that $(\tilde{M},\tilde{D})$ is admissible. Therefore, $\mu_{\tilde{M},\tilde{D}}$ is a spectral measure by Theorem \ref{th(1.2)}.

\textbf{Case 2:} $\tilde{M}\in\mathfrak{M}_3$. Applying \eqref{3.2} and \eqref{3.6}, we have
\begin{equation*}
A_1=Q_1\tilde{M}Q_1^{-1}
=\begin{bmatrix}
3a&9b+3p_2\\
3^{s-1}c&3d
\end{bmatrix}\quad {\rm and} \quad D_1=Q_1\tilde{D}=\left\{\begin{pmatrix}
0\\0\end{pmatrix},\begin{pmatrix}
\sigma\\ 0
\end{pmatrix},
\begin{pmatrix}
\omega\\ 3^{\eta-1}\vartheta
\end{pmatrix}\right\},
\end{equation*}
where $s\geq1$, $a,b,d\in \mathbb{Z}$,  $c\in(\mathbb{Z}\setminus3\mathbb{Z})\cup\{0\}$ and $p_2\in\{1,2\}$.
It follows from $\eta\geq 1$ that $\eta-1\geq 0$. Note that $2 \sigma-\omega \in3\mathbb{Z}$, then Remarks \ref{re(A)} and \ref{re(B)} give that $\mathcal{J}_{0}\setminus\{0\},(\mathcal{J}_{0}-\mathcal{J}_{0})\setminus\{0\}\subset\mathcal{Z}(m_{ D_1})$. Moreover, it is a direct check to see
$A_1^*\mathcal{J}_0\subset \mathbb{Z}^2$.
In view of Lemma \ref{th(2.7)}, $(A_1,D_1)$ is admissible. Hence $\mu_{\tilde{M},\tilde{D}}$ is a spectral measure by Theorem \ref{th(1.2)}  and Lemma \ref{th(2.1)}.

\textbf{Case 3:} $\tilde{M}\in\mathfrak{B}_2$. Form \eqref{3.2}, we can write $\tilde{M}\in\mathfrak{B}_2$ as
$$ \tilde{M}=\begin{bmatrix}
3a+p_1&b'\\
3^{s}c&3d
\end{bmatrix},$$
where $1\leq s<\eta$, $a,b',d\in \mathbb{Z}$, $c\in\mathbb{Z}\setminus3\mathbb{Z}$ and $p_1\in\{1,2\}$.
Applying \eqref{3.6}, we have
\begin{equation*}
A_2=Q_s\tilde{M}Q_s^{-1}
=\begin{bmatrix}
3a+p_1&3^{s}b'\\
c&3d
\end{bmatrix}\quad {\rm and} \quad D_2=Q_s\tilde{D}=\left\{\begin{pmatrix}
0\\0\end{pmatrix},\begin{pmatrix}
\sigma\\ 0
\end{pmatrix},
\begin{pmatrix}
\omega\\ 3^{\eta-s}\vartheta
\end{pmatrix}\right\},
\end{equation*}
As $c\notin 3\mathbb{Z}$ and $\eta-s\geq 1$, one can easily know that $A_2\in \mathfrak{M}_6$ and $D_2$ has the same form as $\tilde{D}$.
By Case 1, we conclude that  $(A_2,D_2)$ is admissible and $\mu_{A_2,D_2}$ is a spectral measure. Therefore,
$\mu_{\tilde{M},\tilde{D}}$ is  a spectral measure by Lemma \ref{th(2.1)}.

Finally,  according to the results of the above three cases, we know that $\mu_{\tilde M,\tilde D}$ is a spectral measure. Hence
the sufficiency follows. This ends the proof of Theorem \ref{th(3.3)}.
\end{proof}

\begin{thm} \label{th(3.4)}
If $2\sigma-\omega \in3\mathbb{Z}$ in $\tilde D$, then $\mu_{\tilde M,\tilde D}$ is a spectral measure if and only if there exists a matrix  $Q\in M_2(\mathbb{Z})$  such that  $(\bar{M},\bar{D})$ is admissible, where $\bar{M}=Q\tilde{M}Q^{-1}$ and $\bar{D}=Q\tilde{D}$.
\end{thm}
\begin{proof}
The sufficiency follows directly from Theorem \ref{th(1.2)} and Lemma \ref{th(2.1)}.
Now we are devoted to proving the necessity. Suppose that $\mu_{\tilde M,\tilde D}$ is a spectral measure, applying Lemma \ref{th(3.2)}, we obtain $\tilde{M}\in\mathfrak{B}_1\cup\mathfrak{B}_2$,  where $\mathfrak{B}_1$ and $\mathfrak{B}_2$ are given by \eqref{3.7} and \eqref{3.8} respectively.  From the proof of  Theorem \ref{th(3.3)}, we have that $(\bar{M},\bar{D})$ is admissible, where
$$\bar{M}=\tilde{M} \quad  {\rm and} \quad \bar{D}=\tilde{D} \quad  {\rm if}\ \tilde{M}\in \mathfrak{B}_1\setminus\mathfrak{M}_3,$$
$$\bar{M}=Q_1\tilde{M}Q_1^{-1} \quad  {\rm and} \quad \bar{D}=Q_1\tilde{D} \quad  {\rm if}\ \tilde{M}\in\mathfrak{M}_3,$$
$$\bar{M}=Q_s\tilde{M}Q_s^{-1} \quad  {\rm and} \quad \bar{D}=Q_s\tilde{D} \quad  {\rm if}\ \tilde{M}\in\mathfrak{B}_2.$$
This proves the necessity, and hence
the proof is completed.
\end{proof}

\section{\bf Proof of Theorem \ref{th(1.4)} in Case II\label{sect.4}}

In this section, we are devoted to proving Theorem \ref{th(1.4)}  in Case II, which is equivalent to proving Theorem \ref{th(2.4)} in the case of  $2 \sigma-\omega \notin3\mathbb{Z}$.
Before starting to prove the theorem, we first characterize the spectral properties of  some special measures $\mu_{\tilde{M},\tilde{D}}$.
Recall  that $\mathfrak{M}_k$ is defined by \eqref{3.2}, where $k\in\{1,2,\ldots,10\}$.

\begin{pro}\label{th(4.1)}
If $\tilde{M}\in\mathfrak{M}_4$ and $2 \sigma-\omega \notin3\mathbb{Z}$ in $\tilde D$, then $\mu_{\tilde{M},\tilde{D}}$ is  a non-spectral measure.
\end{pro}
\begin{proof}
For simplicity, we denote $\mathcal {Z}_j:={\tilde{M}}^{*j}\mathcal {Z}(m_{\tilde{D}})$.  Let
\begin{equation}\label{4.1}
\mu_1=\delta_{\tilde{M}^{-1}\tilde{D}}*\delta_{\tilde{M}^{-3}\tilde{D}}*\cdots  \quad {\rm and} \quad
\mu_{2}=\delta_{\tilde{M}^{-2}\tilde{D}}*\delta_{\tilde{M}^{-3}\tilde{D}}*\cdots.
\end{equation}
Then $\mu_{\tilde{M},\tilde{D}}=\mu_1*\delta_{\tilde{M}^{-2}\tilde{D}}=\mu_{2}*\delta_{\tilde{M}^{-1}\tilde{D}}$. It follows from \eqref{2.4} that
\begin{equation}\label{4.2}
\mathcal{Z}(\hat{\mu}_1)=\bigcup_{j\geq 1, j\neq 2}\mathcal {Z}_j  \quad {\rm and} \quad
\mathcal{Z}(\hat{\mu}_2)=\bigcup_{j\geq 2}\mathcal {Z}_j.
\end{equation}
Let $\Lambda$ be  a bi-zero set of $\mu_{\tilde{M},\tilde{D}}$
with $0\in \Lambda$.

We first claim that $\Lambda\setminus\{0\} \subset \mathcal{Z}(\hat{\mu}_1)$ or $\Lambda\setminus\{0\} \subset \mathcal{Z}(\hat{\mu}_2)$. Otherwise, there exist
$\lambda_1,\lambda_2 \in \Lambda\setminus\{0\}$ such that $\lambda_1\in \mathcal {Z}_{2}$ and $\lambda_2\in \mathcal {Z}_{1}$.
According to the orthogonality of $\Lambda$, we have
\begin{equation}\label{4.3}
\lambda_1-\lambda_2={\tilde{M}}^{*2}\xi_1-{\tilde{M}}^{*}\xi_2={\tilde{M}}^{*\kappa}\xi_3
\end{equation}
for some $\xi_1,\xi_2,\xi_3\in \mathcal {Z}(m_{\tilde{D}})$ and $\kappa\geq 1$.
Multiplying both sides of \eqref{4.3} by $3^\eta\gamma$, we get
\begin{equation}\label{4.4}
{\tilde{M}}^{*2}3^\eta\gamma\xi_1-{\tilde{M}}^{*}3^\eta\gamma\xi_2={\tilde{M}}^{*\kappa}3^\eta\gamma\xi_3.
\end{equation}
Since $\tilde{M}\in\mathfrak{M}_4$, it follows from  \eqref{3.2} that
$${\tilde{M}}^{*}=\begin{bmatrix}
3a&3^{s}c+p_3\\
3b&3d
\end{bmatrix} \quad {\rm and} \quad  {\tilde{M}}^{*2}=3\begin{bmatrix}
3a^2+3^{s}bc+bp_3&(a+d)(3^{s}c+p_3)\\
3b(a+d)&3^{s}bc+3d^2+bp_3
\end{bmatrix}\in M_2(3\mathbb{Z}),$$
where $s\geq1$, $a,b,d\in\mathbb{Z}$,  $c\in(\mathbb{Z}\setminus3\mathbb{Z})\cup\{0\}$ and $p_3\in\{1,2\}$.
Applying $2 \sigma-\omega \notin3\mathbb{Z}$ and Proposition \ref{pro(A)}, we conclude that
\begin{equation}\label{4.5}
{\tilde{M}}^{*}3^\eta\gamma\mathcal {Z}(m_{\tilde{D}})\cap \mathbb{Z}^2=\emptyset  \quad {\rm and} \quad {\tilde{M}}^{*j}3^\eta\gamma\mathcal {Z}(m_{\tilde{D}})\subset \mathbb{Z}^2\ {\rm for \ j\geq 2}.
\end{equation}
This and \eqref{4.4} give $\kappa=1$.  Hence \eqref{4.4} becomes
\begin{equation}\label{4.6}
{\tilde{M}}^{*}3^\eta\gamma(\xi_2+\xi_3)={\tilde{M}}^{*2}3^\eta\gamma\xi_1.
\end{equation}
If $3^\eta\gamma(\xi_2+\xi_3)\in\mathbb{Z}^2$, then \eqref{4.6} implies ${\tilde{M}}^{*}3^\eta\gamma\xi_1\in\mathbb{Z}^2$, which contradicts to \eqref{4.5}.  If $3^\eta\gamma(\xi_2+\xi_3)\notin\mathbb{Z}^2$, it follows from $2 \sigma-\omega \notin3\mathbb{Z}$ and Proposition \ref{pro(A)} that there exist $\ell_{1},\ell_{2},\ell_{3},\ell_{4}\notin3\mathbb{Z}$
and $\ell_{2}+\ell_{4}\notin3\mathbb{Z}$ such that
$\xi_2=( \frac{\ell_{1}}{3\gamma}, \frac{\ell_2}{3^{\eta+1}\gamma})^t$ and $\xi_3=( \frac{\ell_{3}}{3\gamma}, \frac{\ell_4}{3^{\eta+1}\gamma})^t$. Consequently,
$${\tilde{M}}^{*}3^\eta\gamma(\xi_2+\xi_3)=\begin{pmatrix}
  3^\eta a(\ell_{1}+\ell_{3})+3^{s-1} c(\ell_{2}+\ell_{4})+\frac{p_3(\ell_{2}+\ell_{4})}{3} \\ 3^\eta b(\ell_{1}+\ell_{3})+d (\ell_{2}+\ell_{4})
  \end{pmatrix}\notin\mathbb{Z}^2$$
since $s\geq1$ and $p_3,\ell_{2}+\ell_{4}\notin3\mathbb{Z}$.
Together with \eqref{4.5}, it implies that the left-hand side of \eqref{4.6} is not in $\mathbb{Z}^{2}$, while the right-hand side is in  $\mathbb{Z}^{2}$. This is a contradiction, and hence the claim follows.

Next, we prove that $(\Lambda-\Lambda)\setminus\{0\}\subset \mathcal{Z}(\hat{\mu}_1)$ or $(\Lambda-\Lambda)\setminus\{0\}\subset \mathcal{Z}(\hat{\mu}_2)$. According to the claim, we decompose the proof into the following two cases.

If  $\Lambda\setminus\{0\} \subset \mathcal{Z}(\hat{\mu}_1)$, we have $(\Lambda-\Lambda)\setminus\{0\}\subset \mathcal{Z}(\hat{\mu}_1)$. Suppose,
on the contrary, that there exist $\lambda_1, \lambda_2 \in \Lambda\setminus\{0\}$ such  that
\begin{equation}\label{4.7}
\lambda_1-\lambda_2={\tilde{M}}^{*j_1}\xi_1-{\tilde{M}}^{*j_2}\xi_2={\tilde{M}}^{*2}\xi_3,
\end{equation}
where $\xi_1,\xi_2,\xi_3\in \mathcal {Z}(m_{\tilde{D}})$ and $j_1,j_2\in\mathbb{N}\setminus\{2\}$. Similar to \eqref{4.6}, it is  easy to verify that $j_1\neq j_2$. Without loss of generality, we assume $j_1> j_2$. If $j_2=1$, using \eqref{4.5} and $j_1>2$, we have ${\tilde{M}}^{*j_1}3^\eta\gamma\xi_1,{\tilde{M}}^{*2}3^\eta\gamma\xi_3\in\mathbb{Z}^2$ and ${\tilde{M}}^{*}3^\eta\gamma\xi_2\notin\mathbb{Z}^2$. Thus  \eqref{4.7} does not hold. If $j_2>1$, then $j_1>j_2>2$. Multiplying both sides of \eqref{4.7} by ${\tilde{M}}^{*-1}3^\eta\gamma$, we get \begin{equation*}
{\tilde{M}}^{*j_1-1}3^\eta\gamma\xi_1-{\tilde{M}}^{*j_2-1}3^\eta\gamma\xi_2={\tilde{M}}^{*}3^\eta\gamma\xi_3.
\end{equation*}
In view of \eqref{4.5}, the above equation does not hold. Hence $(\Lambda-\Lambda)\setminus\{0\}\subset \mathcal{Z}(\hat{\mu}_1)$.

If $\Lambda\setminus\{0\} \subset \mathcal{Z}(\hat{\mu}_2)$, we have
$(\Lambda-\Lambda)\setminus\{0\}\subset \mathcal{Z}(\hat{\mu}_2)$. The proof is similar to that of $\Lambda\setminus\{0\} \subset \mathcal{Z}(\hat{\mu}_1)$, so we omit it here.

Finally, we show that $\Lambda$ cannot be a spectrum of $\mu_{\tilde{M},\tilde{D}}$. Since $(\Lambda-\Lambda)\setminus\{0\}\subset \mathcal{Z}(\hat{\mu}_1)$ or $(\Lambda-\Lambda)\setminus\{0\}\subset \mathcal{Z}(\hat{\mu}_2)$, it follows from \eqref{2.5} that $\Lambda$ is   a bi-zero set of  $\mu_1$ or  $\mu_2$. Therefore, Lemma \ref{lem(DHL)} means that
$\Lambda$ cannot be a spectrum of $\mu_{\tilde{M},\tilde{D}}$.

This completes the proof of Proposition \ref{th(4.1)}.
\end{proof}

The following lemma is a basic property of the matrix  $\tilde{M}\in\mathfrak{M}_6$.

\begin{lemma}\label{th(4.3)}
If $\tilde{M}\in\mathfrak{M}_6$, then for any $\ell\in \mathbb{N}$, there exist $a_\ell,b_\ell,c_\ell, d_\ell\in\mathbb{Z}$ such that $$\tilde{M}^{*\ell}=\begin{bmatrix}
3a_\ell+p_1^\ell&3c_\ell+p_1^{\ell-1}p_3\\
3b_\ell&3d_\ell
\end{bmatrix},$$
where $p_1,p_3\in\{1,2\}$.
\end{lemma}
\begin{proof}
The proof is similar to that of Lemma \ref{th(3.1)}, so we omit it here.
\end{proof}

Based on Lemma \ref{th(4.3)}, we can further characterize the spectral property of  $\mu_{\tilde M,\tilde D}$ for $\tilde{M}\in\mathfrak{M}_6$.

\begin{pro}\label{th(4.4)}
If $\tilde{M}\in\mathfrak{M}_6$ and $2 \sigma-\omega \notin3\mathbb{Z}$ in $\tilde D$, then $\mu_{\tilde{M},\tilde{D}}$ is  a non-spectral measure.
\end{pro}
\begin{proof}
For any $\tilde{M}\in\mathfrak{M}_6$, we first prove that $\tilde{M}^{*j}\mathcal{Z}_{\tilde D}^2\cap \mathbb{Z}^2=\emptyset$ for $j\in \mathbb{N}$. Suppose, on the contrary, that $\tilde{M}^{*n}\mathcal{Z}_{\tilde D}^2\cap \mathbb{Z}^2\neq\emptyset$ for some $n\in \mathbb{N}$. By $2 \sigma-\omega \notin3\mathbb{Z}$ and Proposition \ref{pro(A)}, we have
$$\mathcal{Z}_{\tilde D}^2\subset \mathcal{G}\cap [0, 1)^2=\left\{
\begin{pmatrix}
  \frac{\ell_{1}}{3\gamma} \\ \frac{\ell_{2}}{3^{\eta+1}\gamma}
  \end{pmatrix}
  :1\leq \ell_{1}\leq 3\gamma-1, 1\leq\ell_{2}\leq 3^{\eta+1}\gamma-1, \ell_{1},\ell_{2}\in\mathbb{Z}\setminus3\mathbb{Z}
  \right\}.$$
Thus there exist $(\frac{\ell_{1}}{3\gamma}, \frac{\ell_{2}}{3^{\eta+1}\gamma})^t\in\mathcal{G}\cap [0, 1)^2$ and $(s_1,s_2)^t\in\mathbb{Z}^2$ such that
\begin{equation}\label{4.8}
\tilde{M}^{*n}\begin{pmatrix}
  \frac{\ell_{1}}{3\gamma} \\ \frac{\ell_{2}}{3^{\eta+1}\gamma}
  \end{pmatrix}=\begin{pmatrix}
  s_{1} \\ s_{2}
  \end{pmatrix}.
\end{equation}
According to Lemma \ref{th(4.3)},  there exist  $a_n,b_n,c_n,d_n\in\mathbb{Z}$ such that $$\tilde{M}^{*n}=\begin{bmatrix}
3a_n+p_1^n&3c_n+p_1^{n-1}p_3\\
3b_n&3d_n
\end{bmatrix},$$
where $p_1,p_3\in\{1,2\}$.
Now multiplying both sides of  \eqref{4.8} by $3^{\eta+1}\gamma$, one may get that
\begin{equation*}
\begin{pmatrix}
 \ell_{2} p_1^{n-1}p_3 \\ d_n\ell_{2}
  \end{pmatrix}=\begin{pmatrix}
  3(3^{\eta}(\gamma s_{1}-a_n\ell_{1})-3^{\eta-1}\ell_{1}p_1^{n}+c_n\ell_{2}) \\ 3^{\eta}(\gamma s_{2}-b_n\ell_{1})
  \end{pmatrix}\in3\mathbb{Z}^2.
\end{equation*}
This is impossible since $\ell_{2}\notin3\mathbb{Z}$ and $p_1,p_3\in\{1,2\}$. Thus $\tilde{M}^{*j}\mathcal{Z}_{\tilde D}^2\cap \mathbb{Z}^2=\emptyset$ for all $j\in \mathbb{N}$. By Theorem \ref{th(Li)},
$\mu_{\tilde{M},\tilde{D}}$ is  a non-spectral measure.
\end{proof}

In Proposition \ref{th(4.4)}, if $\eta=0$ in $\tilde{D}$, it is interesting that for different matrices $\tilde{M}_1,\tilde{M}_2\in\mathfrak{M}_6$, the spectrality of  $\mu_{\tilde{M}_1,\tilde{D}}$ and $\mu_{\tilde{M}_2,\tilde{D}}$ may be different.  The following example is devoted to display the fact.

\begin{ex} \label{th(4.5)}                                                 {\rm Let                                                            \begin{equation*}
\tilde{M}_1=\begin{bmatrix}
 4&0 \\
 1&3
\end{bmatrix}, \quad
\tilde{M}_2=\begin{bmatrix}
   4&0 \\
   2&3
  \end{bmatrix} \quad {\rm and} \quad
\tilde{D}=\left\{\begin{pmatrix}
0\\0\end{pmatrix},\begin{pmatrix}
1\\ 0
\end{pmatrix},
\begin{pmatrix}
0\\ 1
\end{pmatrix}\right\}.
\end{equation*}
Then  $\mu_{\tilde{M}_1,\tilde{D}}$ is a spectral measure, while $\mu_{\tilde{M}_2,\tilde{D}}$ is a non-spectral  measure.}
\end{ex}
\begin{proof}
By \eqref{3.2}, one can easily see that $\tilde{M}_1,\tilde{M}_2\in\mathfrak{M}_6$.
Let
\begin{equation*}
S=\left\{\begin{pmatrix}
0\\0\end{pmatrix},\begin{pmatrix}
2\\ 2
\end{pmatrix},
\begin{pmatrix}
3\\ 1
\end{pmatrix}\right\}
\quad {\rm and} \quad
H=\frac{1}{\sqrt{3}}\left[  e^{2\pi i\langle \tilde{M}_1^{-1}d,
s\rangle}\right]_{d\in \tilde{D}, s\in S}.
\end{equation*}
By a simple calculation, we obtain $H^*H=I$, which implies that
$(\tilde{M}_1,\tilde{D})$ is admissible. Then with Theorem \ref{th(1.2)}, $\mu_{\tilde{M}_1,\tilde{D}}$ is a spectral measure.

Next we prove that $\mu_{\tilde{M}_2,\tilde{D}}$ is a non-spectral  measure. Let
$A=B=I$, it is evident that $A\in GL_2(3)$ and $AB=I\;({\rm{ mod} } \
M_2(3\mathbb{Z}))$.  Then we have
$(A\tilde{M}_2B)^*(1,-1)^t=(2,-3)^t\notin 3\mathbb{Z}^2$. Together with
$\det(\tilde{D})=1\notin 3\mathbb{Z}$ and Theorem  \ref{th(Liu)}, it means that $\mu_{\tilde{M}_2,\tilde{D}}$ is a non-spectral  measure.
\end{proof}

Recall that $\mathfrak{M}_k$ is defined by \eqref{3.2}.
To simplify the notation, in what follows we denote
$$
\mathfrak{R}_1=\left\{\tilde{M}: \tilde{M}\in\mathfrak{M}_k,\ k\in\{4,6,8,10\}\right\},$$
$$\mathfrak{R}_2=\left\{\tilde{M}: \tilde{M}\in\mathfrak{M}_k,\ k\in\{1,2,3,5,7,9\}, \ s<\eta \right\},$$
$$\mathfrak{R}_3=\left\{\tilde{M}: \tilde{M}\in\mathfrak{M}_k,\ k\in\{1,2,3,5,7,9\}, \ s\geq\eta \right\}.$$
To complete the proof of Theorem \ref{th(2.4)} in the case of $2\sigma-\omega \notin3\mathbb{Z}$, we need the following proposition, which indicates the non-spectrality of the measure $\mu_{\tilde M,\tilde D}$ for $\tilde{M}\in\mathfrak{R}_1\cup\mathfrak{R}_2$.
\begin{pro}\label{th(4.6)}
If $\tilde{M}\in\mathfrak{R}_1\cup\mathfrak{R}_2$ and $2 \sigma-\omega \notin3\mathbb{Z}$ in $\tilde D$, then $\mu_{\tilde{M},\tilde{D}}$ is  a non-spectral measure.
\end{pro}
\begin{proof}
We first prove that $\mu_{\tilde M,\tilde D}$ is a non-spectral measure for $\tilde{M}\in\mathfrak{R}_1$.

$(a):$ $\tilde{M}\in\mathfrak{M}_4\cup\mathfrak{M}_6$.
Since $2 \sigma-\omega \notin3\mathbb{Z}$, it follows from Propositions \ref{th(4.1)} and \ref{th(4.4)}  that $\mu_{\tilde{M},\tilde{D}}$ is a non-spectral measure.

$(b):$ $\tilde{M}\in\mathfrak{M}_8\cup\mathfrak{M}_{10}$.
By \eqref{3.2}, one can write $\tilde{M}\in\mathfrak{M}_8\cup\mathfrak{M}_{10}$  as
 $$ \tilde{M}=\begin{bmatrix}
a'&b'\\
3^{s}c+p_3&3d+p_4
\end{bmatrix},$$
where $s\geq1$, $a',b',d\in \mathbb{Z}$, $c\in(\mathbb{Z}\setminus3\mathbb{Z})\cup\{0\}$, $p_3,p_4\in\{1,2\}$ and $a' p_4-b' p_3\in3\mathbb{Z}$.
 Then we can choose  $\tau\in\{1,2\}$ such that $ \tau p_3+p_4\in3\mathbb{Z}$. Combining this with $a' p_4-b' p_3\in3\mathbb{Z}$, we obtain that
\begin{equation}\label{4.9}
(\tau a'+b')p_3=(\tau p_3+p_4)a'-(a' p_4-b' p_3)\in3\mathbb{Z}.
\end{equation}
Notice that $p_3\in\{1,2\}$, so \eqref{4.9} implies  $\tau a'+b'\in3\mathbb{Z}$.
Let $\tilde{Q}=\begin{bmatrix}
1&-\tau\\
0&1
\end{bmatrix}$. Then
\begin{equation*}
\bar{M}_1
=\tilde{Q}\tilde{M}\tilde{Q}^{-1} =\begin{bmatrix}
3^{s}c\tau+a'-\tau p_3 &-3^sc\tau^2-3d\tau+\tau a'+b'-\tau(\tau p_3+p_4)\\
3^{s}c+p_3&3^{s}c\tau +3d+\tau p_3+p_4
\end{bmatrix},
\end{equation*}
and
\begin{equation*}
\bar{D}_1=\tilde{Q}\tilde{D}=\left\{\begin{pmatrix}
0\\0\end{pmatrix},\begin{pmatrix}
\sigma\\ 0
\end{pmatrix},
\begin{pmatrix}
\omega-3^{\eta}\tau\vartheta\\ 3^{\eta}\vartheta
\end{pmatrix}\right\}.
\end{equation*}
As $\tau_2,p_3\in\{1,2\}$ and $ \tau a'+b', \tau_2 p_3+p_4\in3\mathbb{Z}$, it can be easily seen that $\bar{M}_1\in\mathfrak{M}_4$ if $a'-\tau p_3\in3\mathbb{Z}$, and $\bar{M}_1\in\mathfrak{M}_6$ if $a'-\tau p_3\notin3\mathbb{Z}$.  Moreover, it follows from $2 \sigma-\omega \notin3\mathbb{Z}$ and $\eta\geq1$ that $2 \sigma-(\omega-3^{\eta}\tau\vartheta) \notin3\mathbb{Z}$. Thus $\bar{D}_1$ has the similar properties as $\tilde{D}$. Hence Proposition \ref{th(4.1)} or Proposition \ref{th(4.4)}  still holds for $\mu_{\bar{M}_1,\bar{D}_1}$, which implies that $\mu_{\bar{M}_1,\bar{D}_1}$ is a non-spectral measure. Therefore, $\mu_{\tilde{M},\tilde{D}}$ is a non-spectral measure
by Lemma \ref{th(2.1)}.

According to $(a)$ and $(b)$, $\mu_{\tilde M,\tilde D}$ is a non-spectral measure for any $\tilde{M}\in\mathfrak{R}_1$.

Now  we prove that $\mu_{\tilde M,\tilde D}$ is a non-spectral measure for $\tilde{M}\in\mathfrak{R}_2$. It follows from \eqref{3.2} that $\tilde{M}\in\mathfrak{R}_2$
can be expressed as
 $$ \tilde{M}=\begin{bmatrix}
a'&b'\\
3^{s}c&d'
\end{bmatrix},$$
where $1\leq s<\eta$, $a',b',d'\in \mathbb{Z}$, $c\in\mathbb{Z}\setminus3\mathbb{Z}$ and $a'd'\in3\mathbb{Z}$.
By \eqref{3.6}, we obtain that
\begin{equation*}
\bar{M}_2
=Q_s\tilde{M}Q_s^{-1}
=\begin{bmatrix}
a'&3^{s}b'\\
c&d'
\end{bmatrix}\quad {\rm and} \quad \bar{D}_2=Q_s\tilde{D}=\left\{\begin{pmatrix}
0\\0\end{pmatrix},\begin{pmatrix}
\sigma\\ 0
\end{pmatrix},
\begin{pmatrix}
\omega\\ 3^{\eta-s}\vartheta
\end{pmatrix}\right\}.
\end{equation*}
In view of $a'd'\in3\mathbb{Z}$ and $\eta-s\geq 1$, using \eqref{3.2}, one can easily see that $\bar{M}_2\in \mathfrak{R}_1$, and $\bar{D}_2$ has the similar form as $\tilde{D}$. It has been proved that $\mu_{\tilde{M},\tilde{D}}$ is a  non-spectral measure for $\tilde{M}\in \mathfrak{R}_1$, which implies that $\mu_{\bar{M}_2,\bar{D}_2}$ is a  non-spectral measure.
Then by Lemma \ref{th(2.1)}, $\mu_{\tilde{M},\tilde{D}}$ is a  non-spectral measure.

The proof of Proposition \ref{th(4.6)} is now completed.
\end{proof}

We have all ingredients for the proof of Theorem \ref{th(2.4)} in the case of  $2 \sigma-\omega \notin3\mathbb{Z}$. That is to prove
the following theorem:

\begin{thm} \label{th(4.7)}
Let $\tilde D$ and $\tilde M$ be defined by \eqref{2.6} and \eqref{2.7}, respectively. If $2\sigma-\omega \notin3\mathbb{Z}$ in $\tilde D$, then $\mu_{\tilde M,\tilde D}$ is a spectral measure if and only if there exists a matrix  $Q\in M_2(\mathbb{Z})$  such that  $(\bar{M},\bar{D})$ is admissible, where $\bar{M}=Q\tilde{M}Q^{-1}$ and $\bar{D}=Q\tilde{D}$.
\end{thm}
\begin{proof}
We first prove the sufficiency. Suppose that $(\bar{M},\bar{D})$ is admissible, where $\bar{M}=Q\tilde{M}Q^{-1}$ and $\bar{D}=Q\tilde{D}$, by Theorem \ref{th(1.2)}, $\mu_{\bar{M},\bar{D}}$ is a spectral measure. This together with Lemma \ref{th(2.1)} shows that $\mu_{\tilde M,\tilde D}$ is  a spectral measure. Hence the sufficiency follows.

Now we are devoted to proving the necessity. Suppose that $\mu_{\tilde M,\tilde D}$ is a spectral measure, then Proposition \ref{th(4.6)} implies $\tilde M \in\mathfrak{R}_3$.  According to the definition of $\mathfrak{R}_3$ and \eqref{3.2}, we can write the matrix $\tilde{M}\in\mathfrak{R}_3$  as
 $$ \tilde{M}=\begin{bmatrix}
a'&b'\\
3^{s}c&d'
\end{bmatrix},$$
where $s\geq \eta$, $a',b',d'\in \mathbb{Z}$, $c\in(\mathbb{Z}\setminus3\mathbb{Z})\cup\{0\}$ and $a'd'\in3\mathbb{Z}$.
Applying \eqref{3.6}, we get
\begin{equation*}
\bar{M}_3
=Q_\eta\tilde{M}Q_\eta^{-1}
=\begin{bmatrix}
a'&3^{\eta}b'\\
3^{s-\eta}c&d'
\end{bmatrix}\quad {\rm and} \quad \bar{D}_3=Q_\eta\tilde{D}=\left\{\begin{pmatrix}
0\\0\end{pmatrix},\begin{pmatrix}
\sigma\\ 0
\end{pmatrix},
\begin{pmatrix}
\omega\\ \vartheta
\end{pmatrix}\right\}.
\end{equation*}
Since
$\mu_{\tilde M,\tilde D}$ is a spectral measure, it follows from Lemma \ref{th(2.1)} that $\mu_{\bar{M}_3,\bar{D}_3}$ is also a spectral measure.
Together with $\sigma\vartheta\notin 3\mathbb{Z}$ and Theorem \ref{th(Liu)}, it gives that $(\bar{M}_3,\bar{D}_3)$ is admissible.

This completes the proof of Theorem \ref{th(4.7)}.
\end{proof}

\begin{re}\label{th(4.8)}
{\rm
Under the assumption of  $2\sigma-\omega \notin3\mathbb{Z}$ in $\tilde D$, from the proof of Theorem \ref{th(2.4)} and Theorem \ref{th(Liu)},  we can easily see that $\mu_{\tilde M,\tilde D}$ is a spectral measure if and only if  $s\geq \eta$ and $(AQ_\eta\tilde{M}Q_\eta^{-1}B)^*(1,-1)^t\in 3\mathbb{Z}^2$, where $B=\begin{bmatrix}
\sigma &\omega\\
0&\vartheta
\end{bmatrix}$ and $A\in GL_2(3)$ satisfies $AB=I\;({\rm{ mod} } \ M_2(3\mathbb{Z}))$. By noting that $\sigma, \vartheta\notin 3\mathbb{Z}$, it is clear that $A=\sigma\vartheta\begin{bmatrix}
\vartheta &-\omega\\
0&\sigma
\end{bmatrix}$ satisfies $AB=(\sigma\vartheta)^2I=I\;({\rm{ mod} } \ M_2(3\mathbb{Z}))$.}
\end{re}

Now we give an example to illustrate the fact in the above remark.

\begin{ex} \label{th(4.9)}
{\rm Let
\begin{equation*}
M_1=\begin{bmatrix}
 8&-5 \\
 4&-1
\end{bmatrix}, \quad
M_2=\begin{bmatrix}
5&-1\\
2&2
  \end{bmatrix} \quad {\rm and} \quad
D=\left\{\begin{pmatrix}
0\\0\end{pmatrix},\begin{pmatrix}
2\\ 1
\end{pmatrix},
\begin{pmatrix}
2\\ 4
\end{pmatrix}\right\}.
\end{equation*}
Then  $\mu_{M_1,D}$ is a spectral measure, while $\mu_{M_2,D}$ is a non-spectral  measure.}
\end{ex}
\begin{proof}
Let $P=\begin{bmatrix}
1 &-1\\
-1&2
\end{bmatrix}$. Then we obtain that
\begin{equation*}
\tilde{M}_1=PM_1P^{-1}=\begin{bmatrix}
 4&0 \\
 3&3
\end{bmatrix}, \quad
\tilde{M}_2=PM_2P^{-1}=\begin{bmatrix}
  3&0 \\
  3&4
  \end{bmatrix},
\end{equation*}
and
\begin{equation*}
\tilde{D}=PD=\left\{\begin{pmatrix}
0\\0\end{pmatrix},\begin{pmatrix}
1\\ 0
\end{pmatrix},
\begin{pmatrix}
-2\\ 6
\end{pmatrix}\right\}.
\end{equation*}
According to \eqref{2.6}, it is easy to see that $\sigma=\eta=1$, $\omega=-2$ and $\vartheta=2$.
Therefore,  $A,B$ and $Q_\eta$ in Remark \ref{th(4.8)} can be written  as
\begin{equation*}
A=\begin{bmatrix}
4 &4\\
0&2
\end{bmatrix},\quad B=\begin{bmatrix}
1 &-2\\
0&2
\end{bmatrix} \quad {\rm and} \quad Q_\eta=\begin{bmatrix}
 1&0 \\
 0&\frac{1}{3}
\end{bmatrix}.
\end{equation*}
By calculations, we get
\begin{equation*}
(AQ_\eta\tilde{M}_1Q_\eta^{-1}B)^*\begin{pmatrix}
1\\-1\end{pmatrix}=\begin{bmatrix}
20&2 \\
-16&8
\end{bmatrix}\begin{pmatrix}
1\\-1\end{pmatrix}=\begin{pmatrix}
18\\-24\end{pmatrix}\in 3\mathbb{Z}^2,
\end{equation*}
and
\begin{equation*}
(AQ_\eta\tilde{M}_2Q_\eta^{-1}B)^*\begin{pmatrix}
1\\-1\end{pmatrix}=\begin{bmatrix}
16&2 \\
0&12
\end{bmatrix}\begin{pmatrix}
1\\-1\end{pmatrix}=\begin{pmatrix}
14\\-12\end{pmatrix}\notin 3\mathbb{Z}^2.
\end{equation*}
Then by using Theorem  \ref{th(Liu)}, one can conclude that $\mu_{Q_\eta\tilde{M}_1Q_\eta^{-1},Q_\eta\tilde{D}}$ is a spectral measure and $\mu_{Q_\eta\tilde{M}_2Q_\eta^{-1},Q_\eta\tilde{D}}$ is a non-spectral  measure. Hence the assertion follows by Lemma \ref{th(2.1)}.
\end{proof}

\begin{re} \label{th(4.10)}
{\rm Form  the proof of Example \ref{th(4.9)}, we know that $(Q_\eta\tilde{M}_1Q_\eta^{-1},Q_\eta\tilde{D})$ is admissible. However, $(\tilde{M}_1,\tilde{D})$ is not admissible since $\mathcal {Z}(m_{\tilde{D}})\subset\mathfrak{G}:=\{(\ell_1/3, \ell_2/18)^t: \ell_1,\ell_2\notin 3\mathbb{Z}\}$ and $\tilde{M}_1^*\mathfrak{G}\cap\mathbb{Z}^2=\emptyset$.
This means that the matrix $Q$  is nontrivial in
Theorem \ref{th(1.4)}.
}
\end{re}

At the end of this paper, we give some remarks and an open problem which related to our main results. In the paper, the integer digit set $D=\{(0,0)^t,(\alpha_1,\alpha_2)^t,(\beta_1,\beta_2)^t\}$ satisfies $\alpha_1\beta_2-\alpha_2\beta_1\neq 0$, it is of interest to consider the following question:

{\bf(Qu 2):} For an expanding matrix $M\in M_2(\mathbb{Z})$  and the digit set $D$ given by \eqref{1.2} with $\alpha_1\beta_2-\alpha_2\beta_1=0$, what is the sufficient and necessary condition for $\mu_{M,D}$ to be a spectral measure?

It is  obvious that $\alpha_1\beta_2-\alpha_2\beta_1=0$ implies the digit set $D$ is collinear, i.e., $D=\{0,a,b\}v$ for some $a,b\in \mathbb{Z}$ and $v\in \mathbb{Z}^2$. At this time,
it is easy to show that $\mathcal{ Z}(m_D)\neq \emptyset$ if and only if $\{a,b\}=\{1,2\}\pmod 3$. If  $\{a,b\}=\{1,2\}\pmod 3$, by using the similar proof as Theorem 1.2 in \cite{Liu-Luo_2017},
we can prove that $\mu_{M,D}$ is a spectral measure if $\det(M)\in3\mathbb{Z}$. However,  we know nothing about the necessary condition. We conjecture that  $\det(M)\in3\mathbb{Z}$  is also the  necessary condition for  $\mu_{M,D}$ to be a spectral measure.

\end{document}